\documentclass[a4paper]{amsart}
\usepackage[english]{babel}
\usepackage{amsmath,amsthm,amssymb,amsfonts}

\newtheorem{thm}{Theorem}
\newtheorem{prp}[thm]{Proposition}
\newtheorem{cor}[thm]{Corollary}
\newtheorem{lem}[thm]{Lemma}

\makeatletter
\newcommand\textsubscript[1]{\@textsubscript{\selectfont#1}}
\def\@textsubscript#1{{\m@th\ensuremath{_{\mbox{\fontsize\sf@size\z@#1}}}}}
\makeatother

\DeclareMathOperator{\Kern}{\mathcal{K}}
\DeclareMathOperator{\supp}{\mathrm{supp}}
\DeclareMathOperator{\tr}{\mathrm{tr}}

\DeclareMathOperator{\sgn}{\mathrm{sgn}}

\newcommand{\N}{\mathbb{N}}
\newcommand{\Z}{\mathbb{Z}}
\newcommand{\R}{\mathbb{R}}
\newcommand{\C}{\mathbb{C}}

\newcommand{\tc}{\,:\,}

\newcommand{\leftopenint}{\left]}
\newcommand{\rightopenint}{\right[}
\newcommand{\leftclosedint}{\left[}
\newcommand{\rightclosedint}{\right]}

\DeclareMathOperator{\pf}{\mathrm{Pf}}

\newcommand{\vecU}{\mathbf{U}}

\newcommand{\Ell}{\mathcal{L}}

\newcommand{\lie}{\mathfrak}
\newcommand{\fst}{\lie{g}_1}
\newcommand{\snd}{\lie{g}_2}
\newcommand{\ddsnd}{\lie{g}^*_{2,\mathrm{r}}}
\newcommand{\dusnd}{\snd^*}

\newcommand{\cone}{\mathrm{\Gamma}}

\newcommand{\cnnb}{\mathrm{c}}
\newcommand{\prnb}{\mathrm{p}}
\newcommand{\rad}{\mathrm{r}}
\newcommand{\sph}{\mathrm{s}}

\newcommand{\radp}{\rho}
\newcommand{\crdp}{\delta}
\newcommand{\thinp}{\epsilon}
\newcommand{\smct}{a}
\newcommand{\smctb}{\sigma}

\newcommand{\done}{M}
\newcommand{\dtwo}{{d_2}}
\newcommand{\jone}{j}
\newcommand{\jtwo}{k}

\newcommand{\polbul}{*}

\newcommand{\defeq}{\mathrel{:=}}

\newcommand{\SC}{\mathrm{SC}}
\newcommand{\BD}{\mathrm{BD}}
\newcommand{\HDP}{\mathrm{HDP}}

\begin{document}
\title[Spectral multipliers on $2$-step groups]{Spectral multiplier theorems of Euclidean type on new classes of $2$-step stratified groups}
\author{Alessio Martini}
\address{Alessio Martini \\ Mathematisches Seminar \\ Christian-Albrechts-Universit\"at zu Kiel \\ Ludewig-Meyn-Str.\ 4 \\ D-24118 Kiel \\ Germany}
\email{martini@math.uni-kiel.de}
\author{Detlef M\"uller}
\address{Detlef M\"uller \\ Mathematisches Seminar \\ Christian-Albrechts-Universit\"at zu Kiel \\ Ludewig-Meyn-Str.\ 4 \\ D-24118 Kiel \\ Germany}
\email{mueller@math.uni-kiel.de}

\thanks{The first-named author gratefully acknowledges the support of the Alexander von Humboldt Foundation.}

\begin{abstract}
From a theorem of Christ and Mauceri and Meda it follows that, for a homogeneous sublaplacian $L$ on a $2$-step stratified group $G$ with Lie algebra $\lie{g}$, an operator of the form $F(L)$ is of weak type $(1,1)$ and bounded on $L^p(G)$ for $1 < p < \infty$ if the spectral multiplier $F$ satisfies a scale-invariant smoothness condition of order $s > Q/2$, where $Q = \dim \lie{g} + \dim[\lie{g},\lie{g}]$ is the homogeneous dimension of $G$. Here we show that the condition can be pushed down to $s > d/2$, where $d = \dim \lie{g}$ is the topological dimension of $G$, provided that $d \leq 7$ or $\dim [\lie{g},\lie{g}] \leq 2$.
\end{abstract}

\maketitle

\section{Introduction}

Let $G$ be a $2$-step stratified group. In other words, $G$ is a connected, simply connected nilpotent Lie group, whose Lie algebra $\lie{g}$ is endowed with a decomposition $\lie{g} = \fst \oplus \snd$ for some nontrivial subspaces $\fst,\snd$ of $\lie{g}$, called layers, such that $[\fst,\fst] = \snd$ and $[\lie{g},\snd] = \{0\}$. Let $L$ be a homogeneous sublaplacian on $G$, that is, a second-order, left-invariant differential operator of the form $L = -\sum_j X_j^2$, where $\{X_j\}_j$ is a basis of the first layer $\fst$. Since $L$ is (essentially) self-adjoint on $L^2(G)$, a functional calculus for $L$ is defined via the spectral theorem and, for all bounded Borel functions $F : \R \to \C$, the operator $F(L)$ is $L^2$-bounded.

For the $L^p$-boundedness for $p\neq 2$ of an operator of the form $F(L)$, sufficient conditions can be given in terms of smoothness assumptions on the spectral multiplier $F$. Namely, let $Q = \dim \fst+2 \dim\snd$ be the homogeneous dimension of $G$, denote by $W_2^s$ the $L^2$ Sobolev space of (fractional) order $s$, and let $\chi \in C^\infty_c(\leftopenint 0,\infty \rightopenint)$ be nontrivial. Then the following well-known result holds.

\begin{thm}[Christ \cite{christ_multipliers_1991}, Mauceri and Meda \cite{mauceri_vectorvalued_1990}]\label{thm:christmaucerimeda}
If
\[
\sup_{t > 0} \| F(t \cdot) \, \chi \|_{W_2^s} < \infty
\]
for some $s > Q/2$, then the operator $F(L)$ is of weak type $(1,1)$ and bounded on $L^p(G)$ for all $p \in \leftopenint 1,\infty \rightopenint$.
\end{thm}

This theorem holds in fact for a stratified group of arbitrary step, but here we are interested only in the $2$-step case. Our aim is to improve Theorem~\ref{thm:christmaucerimeda}, by pushing down the smoothness condition to $s > d/2$, where $d = \dim\fst+\dim\snd$ is the topological dimension of $G$. Such an improvement of Theorem~\ref{thm:christmaucerimeda}
(corresponding in the Euclidean setting to the classical Mihlin-H\"ormander theorem for the Laplacian $L = -\Delta$ on $\R^d$)
is known to hold for specific classes of groups \cite{mller_spectral_1994,hebisch_multiplier_1993,martini_n32,martini_heisenbergreiter}, including the groups of Heisenberg type and more generally the M\'etivier groups,
but it is still an open problem whether it may be achieved for an arbitrary $2$-step group. The main result of the present paper reads as follows.

\begin{thm}\label{thm:mhlow}
Suppose that $\dim\snd \leq 2$ or $d \leq 7$. If
\begin{equation}\label{eq:mh}
\sup_{t > 0} \| F(t \cdot) \, \chi \|_{W_2^s} < \infty
\end{equation}
for some $s > d/2$, then the operator $F(L)$ is of weak type $(1,1)$ and bounded on $L^p(G)$ for all $p \in \leftopenint 1,\infty \rightopenint$.
\end{thm}

Necessary and sufficient conditions for the $L^p$-boundedness of operators belonging to the functional calculus of a (sub)elliptic operator $L$ have been extensively studied in many contexts. In some cases (e.g., for Laplace-Beltrami operators on Riemannian symmetric spaces of the non-compact type, and for some sublaplacians on Lie groups of exponential growth) it is known that the $L^p$-boundedness of $F(L)$ for some $p \neq 2$ implies the existence of a holomorphic extension of $F$ to a complex neighborhood of a nonisolated point of the $L^2$-spectrum of $L$ \cite{anker_fourier_1990,anker_multiplicateurs_1986,christ_spectral_1996,clerc_multipliers_1974,hebisch_sub-laplacians_2005,ludwig_sub-laplacians_2000,ludwig_holomorphic_2008,taylor_estimates_1989}. In contrast, in other cases (e.g., for sublaplacians on Lie groups of polynomial growth, and more generally for operators with Gaussian-type heat kernel bounds on spaces of homogeneous type), a condition of the form \eqref{eq:mh}, requiring only a finite order of differentiability $s$ on $F$, is sufficient to guarantee the $L^p$-boundedness of $F(L)$ for $p \in \leftopenint 1,\infty \rightopenint$, provided $s$ is sufficiently large \cite{alexopoulos_spectral_1994,duong_plancherel-type_2002,hebisch_functional_1995}. In this context, several works have been devoted to obtaining, for some particular spaces and operators, the same smoothness condition as in the Euclidean case, i.e., $s > d/2$, where $d$ is the topological dimension of the space (see, e.g., \cite{cowling_spectral_2001,cowling_spectral_2011,martini_grushin2,martini_grushin_2012} for works outside the realm of stratified groups).

Let us return to the initial setting of a homogeneous sublaplacian $L$ on a $2$-step stratified group $G$. The proof of Theorem~\ref{thm:mhlow} is reduced, by a standard argument (see, e.g., \cite[Theorem~4.6]{martini_joint_2012}) based on the Calder\'on-Zygmund theory of singular integral operators, to the following $L^1$-estimate for the convolution kernel $\Kern_{F(L)}$ of the operator $F(L)$ corresponding to a compactly supported multiplier $F$.

\begin{prp}\label{prp:l1low}
Suppose that $\dim\snd \leq 2$ or $d \leq 7$. For all compact sets $K \subseteq \R$, for all Borel functions $F : \R \to \C$ with $\supp F \subseteq K$, and for all $s > d/2$,
\[
\|\Kern_{F(L)}\|_{1} \leq C_{K,s} \|F\|_{W_2^s}.
\]
\end{prp}

Estimates of this kind on an arbitrary stratified group have been known for a long time \cite[Lemma 6.35]{folland_hardy_1982} under a stronger assumption on $s$.
In particular, in \cite[Lemma 1.2]{mauceri_vectorvalued_1990} this $L^1$-estimate is proved\footnote{In \cite[Lemma 1.2]{mauceri_vectorvalued_1990} the $L^1$-estimate and the weighted $L^2$-estimate \eqref{eq:standardl2} are stated under the hypothesis that the compact set $K$ does not contain $0$. Nevertheless an inspection of the proof (based on \cite{de_michele_mulipliers_1987}) shows that for a nonabelian stratified group $G$ this restriction on $K$ can be removed. See also the discussion in \cite{mller_spectral_1994}, where even the abelian case is covered.} for $s > Q/2$, as a consequence of a weighted $L^2$-estimate: if $|\cdot|_G : G \to \R$ is a homogeneous norm on $G$, then, for all $\beta > \alpha \geq 0$ and all multipliers $F : \R \to \C$ supported in a compact set $K \subseteq \R$,
\begin{equation}\label{eq:standardl2}
\|(1+|\cdot|_G)^\alpha \Kern_{F(L)}\|_{2} \leq C_{K,\alpha,\beta} \|F\|_{W_2^\beta}.
\end{equation}
The known improvements of Theorem~\ref{thm:christmaucerimeda} are all based on an improved version of \eqref{eq:standardl2} entailing an ``extra weight'' $w : G \to \leftclosedint 1,\infty \rightopenint$, i.e.,
\begin{equation}\label{eq:weightedl2}
\|(1+|\cdot|_G)^\alpha \, w \, \Kern_{F(L)}\|_{2} \leq C_{K,\alpha,\beta} \|F\|_{W_2^\beta}.
\end{equation}
Different types of weights $w$ are used in the various works \cite{mller_spectral_1994,hebisch_multiplier_1993,martini_n32,martini_heisenbergreiter}; in particular, \cite{hebisch_multiplier_1993} uses an extra weight depending (in exponential coordinates) only on the variables on the first layer, whereas \cite{martini_n32,martini_heisenbergreiter} exploit a weight depending only on the variables on the second layer. In any case, the presence of the extra weight is sufficient to compensate the difference $\dtwo = \dim \snd$ between the homogeneous dimension and the topological dimension.

In the present work, however, no ``global'' $L^2$-estimate of the form \eqref{eq:weightedl2} is obtained. More precisely, if $U_1,\dots,U_\dtwo$ is a basis of $\snd$, then the sublaplacian $L$ and the ``central derivatives'' $-iU_1,\dots,-iU_\dtwo$ admit a joint functional calculus. If $\vecU$ denotes the ``vector of operators'' $(-iU_1,\dots,-iU_\dtwo)$, then, by the use of a suitable partition of unity $\{\zeta_\iota\}_\iota$ here we decompose the operator $F(L)$ along the spectrum of $\vecU$, thus
\begin{equation}\label{eq:decomposition}
\Kern_{F(L)} = \sum_\iota \Kern_{F(L) \, \zeta_\iota(\vecU)}.
\end{equation}
For each piece $\Kern_{F(L) \, \zeta_\iota(\vecU)}$ we prove a weighted $L^2$-estimate of the type \eqref{eq:weightedl2}, where the extra weight $w$ may depend on the piece, hence these estimates cannot be directly summed; however they can be summed at the level of $L^1$, after the application of H\"older's inequality, thus yielding the improved $L^1$-estimate of Proposition~\ref{prp:l1low}.

The decomposition \eqref{eq:decomposition} is related to the possible ``singularities'' of the algebraic structure of $G$. Namely, let $\langle \cdot, \cdot \rangle$ be the inner product on $\fst$ determined by the sublaplacian, and define for all $\eta \in \dusnd$ the skewsymmetric endomorphism $J_\eta$ of $\fst$ by
\[
\langle J_\eta x, x' \rangle = \eta([x,x']) \qquad\text{for all $x,x'\in \fst$.}
\]
Then $-J_\eta^2$ can be decomposed according to the spectral theorem, i.e.,
\[
-J_\eta^2 = \sum_{\jone=1}^\done (b_\jone^\eta)^2 P_\jone^\eta
\]
for some distinct $b_1^\eta,\dots,b_\done^\eta \in \leftopenint 0,\infty \rightopenint$ and some projections $P_1^\eta,\dots,P_\done^\eta$ on mutually orthogonal subspaces of $\fst$ of even ranks. By the use of the representation theory of the nilpotent group $G$, a formula for the (Euclidean) Fourier transform $\widehat \Kern_{H(L,\vecU)}$ of the convolution kernel of an operator $H(L,\vecU)$
 in the joint functional calculus 
can be written, involving the quantities $b_1^\eta,\dots,b_\done^\eta$, $P_1^\eta,\dots,P_\done^\eta$. Weighted $L^2$-estimates of $\Kern_{H(L,\vecU)}$ correspond, roughly speaking, to $L^2$-estimates of derivatives of $\widehat \Kern_{H(L,\vecU)}$; therefore we are interested in controlling the derivatives of the (algebraic) functions $\eta \mapsto b_\jone^\eta$ and $\eta \mapsto P_\jone^\eta$.

The singularities of these functions lie on a homogeneous Zariski-closed subset of $\dusnd$. For the groups considered in \cite{mller_spectral_1994,martini_n32,martini_heisenbergreiter}, the only relevant singularity is at the origin of $\dusnd$, and the derivatives of the $b_\jone^\eta$ and $P_\jone^\eta$ can be simply controlled by homogeneity. This is not the case for more general $2$-step groups. Nevertheless, if $\dim \snd = 2$, then the singular set is a finite union of rays emanating from the origin; by the use of a finite decomposition \eqref{eq:decomposition} we can consider each of these rays separately, and classical results for the resolution of singularities of algebraic curves allow us to obtain the desired estimate.

For the case $d \leq 7$, it remains then to consider some examples where $\dim \snd = 3$. It turns out that in most examples the singular set is a finite union of lines and planes, and an adaptation of the technique used when $\dim \snd = 2$ works here too. However there is an (essentially unique) example where the singular set has a nonflat component, namely a conic surface. In this case, in the neighborhood of the cone we exploit an infinite decomposition \eqref{eq:decomposition} analogous to the ``second dyadic decomposition'' used in \cite{seeger_regularity_1991} to prove sharp $L^p$ estimates for Fourier integral operators. Due to the ``too large amount'' of pieces, this technique alone would give only a partial improvement of Theorem~\ref{thm:christmaucerimeda}; however a further extra weight can be gained in this case by a variation of the technique of \cite{hebisch_multiplier_1993,hebisch_multiplier_1995} (as extended in \cite[\S 3]{martini_joint_2012} to the joint functional calculus of commuting operators), and the combination of the two techniques yields eventually the wanted result.

The rest of this paper is devoted to the proof of Proposition~\ref{prp:l1low}. Namely, in \S\ref{section:2dc} below the case $\dim\snd \leq 2$ is considered, while in \S\ref{section:7dim} and \S\ref{section:g743} we deal with the groups of dimension at most $7$. First, however, in \S\ref{section:kernel} we obtain the formula for $\widehat \Kern_{H(L,\vecU)}$, and in \S\ref{section:kernelderivatives} we develop a technique to deal with derivatives of this formula; these preliminary results are valid on all $2$-step groups.

\section{Joint functional calculus and kernel formula}\label{section:kernel}

Let $G$ be a connected, simply connected $2$-step nilpotent Lie group. Let $\lie{g}$ be its Lie algebra, and let $\lie{g} = \lie{g}_1 \oplus \lie{g}_2$ be a stratification of $\lie{g}$; in other words, $[\lie{g}_1,\lie{g}_1] = \lie{g}_2 \neq \{0\}$ and $[\lie{g},\lie{g}_2] = \{0\}$, and in particular $\lie{g}_2= [\lie{g},\lie{g}]$ is contained in the center of $\lie{g}$. The group $G$ can be identified with its Lie algebra $\lie{g}$ via the exponential map, and the Haar measure on $G$ corresponds to the Lebesgue measure on $\lie{g}$. In particular, an element of $G$ can be written as $(x,u)$, where $x$ and $u$ denote the components in $\fst$ and $\snd$ respectively. A homogeneous norm $|\cdot|_G$ on $G$ is then defined by
\begin{equation}\label{eq:homogeneousnorm}
|(x,u)|_G = |x| + |u|^{1/2}
\end{equation}
for any choice of norms on $\fst$ and $\snd$ (see \cite[\S1.2]{goodman_nilpotent_1976} for a more general definition of homogeneous norm).

A homogeneous sublaplacian $L$ on $G$ is an operator of the form $- \sum_j X_j^2$ for some basis $\{X_j\}_j$ of the first layer $\lie{g}_1$. A homogeneous sublaplacian $L$ on $G$ determines uniquely an inner product $\langle \cdot, \cdot \rangle$ on $\lie{g}_1$ so that $L = - \sum_j \tilde X_j^2$ for any orthonormal basis $\{\tilde X_j\}_j$ of $\lie{g}_1$; vice versa, an inner product $\langle \cdot,\cdot \rangle$ on $\lie{g}_1$ determines a homogeneous sublaplacian $L$.

Let $\dtwo = \dim \snd$, and let $\{U_k\}_k$ be a basis of $\snd$. Then the operators
\begin{equation}\label{eq:operators}
L,-iU_1,\dots,-iU_\dtwo
\end{equation}
are essentially self-adjoint and commute strongly on $L^2(G)$, hence they admit a joint functional calculus (see, e.g., \cite[Corollary~3.3]{martini_spectral_2011}). Denote by $\vecU$ the ``vector of operators'' $(-iU_1,\dots,-iU_\dtwo)$. If $\dusnd$ is identified with $\R^\dtwo$ via the chosen basis $\{U_k\}_k$ of $\snd$, then the operator $H(L,\vecU)$ is well-defined and bounded on $L^2(G)$ for all bounded Borel functions $H : \R \times \dusnd \to \C$. Since $L,-iU_1,\dots,-iU_\dtwo$ are left-invariant, the same holds for $H(L,\vecU)$, and we denote by $\Kern_{H(L,\vecU)}$ its convolution kernel.

For all $\eta \in \dusnd$, define $J_\eta$ as the unique endomorphism of $\fst$ such that
\[\langle J_\eta x,x' \rangle = \eta([x,x'])\]
for all $x,x'\in \fst$. Note that $J_\eta$ is skewadjoint for all $\eta \in \dusnd$, hence $-J_\eta^2 = J_\eta^* J_\eta$ is selfadjoint and nonnegative. Let $p_\eta$ be the characteristic polynomial of $-J_\eta^2$, i.e.,
\begin{equation}\label{eq:charpoly}
p_\eta(\lambda) = \det (\lambda + J_\eta^2).
\end{equation}
We show now that the polynomials $p_\eta$ admit a ``simultaneous factorization'' when $\eta$ ranges in a Zariski-open subset of $\dusnd$. This is in fact a classical result based on the theory of discriminants and resultants (for which we refer the reader, e.g., to \cite[\S 3.5]{cox_ideals_1997}, \cite[\S A.1]{fischer_plane_2001}, and references therein), nevertheless we sketch a proof here for completeness.

\begin{lem}\label{lem:factorization}
There exist a nonempty, homogeneous Zariski-open subset $\ddsnd$ of $\dusnd$ and numbers $\done \in \N\setminus \{0\}$, $r_0 \in \N$, $r_1,\dots,r_\done \in \N \setminus \{0\}$ such that
\[
p_\eta(\lambda) = \lambda^{r_0} (\lambda - (b_1^\eta)^2)^{2r_1} \cdots (\lambda - (b_\done^\eta)^2)^{2r_\done}
\]
for all $\eta \in \dusnd$, where the $\eta \mapsto b_j^\eta$ are continuous functions on $\dusnd$ and real analytic functions on $\ddsnd$, homogeneous of degree $1$, such that
\[
b_j^\eta > 0 \qquad\text{ and }\qquad b_j^\eta \neq b_{j'}^\eta \,\text{ if $j \neq j'$}
\]
for all $\eta \in \ddsnd$ and $j,j' \in \{1,\dots,\done\}$.
\end{lem}
\begin{proof}
For all $\eta \in \dusnd$, the roots of $p_{\eta}$ are the eigenvalues of $-J_{\eta}^2$, which are all real and nonnegative, and moreover the nonzero eigenvalues have even multiplicity, since they correspond to pairs of conjugate eigenvalues of $J_\eta$. What remains to show is essentially that the number and the multiplicities of the roots of $p_\eta$ do not change when $\eta$ ranges in a nonempty homogeneous Zariski-open subset of $\dusnd$, and that the roots are real analytic functions of $\eta$ there. 

Since $J_\eta$ is a linear function of $\eta$, the coefficients of $p_\eta$ are polynomial functions of $\eta$; hence $\eta \mapsto p_\eta$ can be identified with an element $p_\polbul$ of $\R[\dusnd][\lambda]$, that is, with a polynomial in the indeterminate $\lambda$ whose coefficients are polynomials on $\dusnd$.

$\R[\dusnd][\lambda]$ is a unique factorization domain, hence we can write
\begin{equation}\label{eq:factorization}
p_\polbul = \lambda^{s_0} p_{1,\polbul}^{s_1} \cdots p_{n,\polbul}^{s_n}
\end{equation}
where $s_0 \in \N$, $n,s_1,\dots,s_n \in \N \setminus \{0\}$, and the $p_{l,\polbul}$ are monic and irreducible elements of $\R[\dusnd][\lambda]$, pairwise coprime and coprime with $\lambda$.

Suppose first that $n = 1$. Let $g$ be the degree (in $\lambda$) of $p_{1,\polbul}$ and, for all $\eta \in \dusnd$, let $R_{1}^\eta \leq \dots \leq R_{g}^\eta$ be the increasing enumeration of the roots of $p_{1,\eta}$, repeated according to their multiplicities. By Rouch\'e's theorem, the $\eta \mapsto R_j^\eta$ are continuous on $\dusnd$. Since $p_{1,\polbul}$ is monic, irreducible and not divisible by $\lambda$, its ``constant term'' $C_{\polbul} = p_{1,\polbul}(0)$ and its discriminant $D_\polbul$ are nonzero elements of $\R[\dusnd]$. For all $\eta \in \dusnd$, if $D_{\eta} \neq 0$, then the roots of $p_{1,\eta}$ are simple. Therefore, if we set $\ddsnd = \{ \eta \tc D_{\eta} \cdot C_{\eta} \neq 0\}$, then, for all $\eta \in \ddsnd$, the $R_{l}^\eta$ are nonzero and distinct, and do not annihilate $\partial_\lambda p_{1,\eta}$, and in particular they are analytic functions of $\eta \in \ddsnd$ by the implicit function theorem. Moreover, by \eqref{eq:factorization}, $R_j^\eta$ is a root of $p_{\eta}$ of multiplicity $s_1$ for all $\eta \in \ddsnd$.

Suppose instead that $n > 1$. Then,  proceeding as before, for each irreducible factor $p_{l,\polbul}$ of $p_\polbul$ we find a system of nonnegative continuous functions $\eta\mapsto R_{l,j}^\eta$ ($j=1,\dots,g_l$) such that
\begin{equation}\label{eq:factorroots}
p_{l,\eta}(\lambda) = \prod_j (\lambda-R_{l,j}^\eta), \qquad R_{l,1}^\eta \leq \dots \leq R_{l,g_l}^\eta,
\end{equation}
for all $\eta \in \dusnd$, and moreover we find a Zariski-open set $A_l$ such that, for all $\eta \in A_l$, the quantities $R_{j,1}^\eta,\dots,R_{j,g_j}^\eta$ are nonzero and distinct, and analytic functions of $\eta\in A_l$. In particular, by \eqref{eq:factorization},
\[
p_\eta(\lambda) = \lambda^{s_0} \prod_l \prod_j (\lambda-R_{l,j}^\eta)^{s_l}.
\]
It is however possible that roots $R_{l,j}^\eta$ and $R_{l',j'}^\eta$ coming from two distinct factors $p_{l,\polbul}$ and $p_{l',\polbul}$ coincide for some $\eta \in A_1 \cap \dots \cap A_n$.
To circumvent this, we consider the resultant $S_{l,l',\polbul}$ of $p_{l,\polbul}$ and $p_{l',\polbul}$, which is a nonzero element of $\R[\dusnd]$, because $p_{l,\polbul}$ and $p_{l',\polbul}$ are coprime. By setting $\ddsnd = \bigcap_l A_l  \cap \bigcap_{l\neq l'} \{ \eta \tc S_{l,l',\eta} \neq 0\}$, we obtain that the $R_{l,j}^\eta$ are all distinct and nonzero for all $\eta \in \ddsnd$, hence $R_{l,j}^\eta$ is a root of $p_\eta$ of multiplicity $s_l$ for all $\eta \in \ddsnd$.

It remains to discuss the homogeneity of $\ddsnd$ and the $\eta \mapsto R_{l,j}^\eta$. Note that the function $(\eta,t) \mapsto p_\eta(t^2)$ is homogeneous; in other words, if we define an algebra gradation $\gamma$ on $\R[\dusnd][\lambda]$ by assigning the standard polynomial degree to the elements of $\R[\dusnd]$ and degree $2$ to $\lambda$, then $p_\polbul$ is $\gamma$-homogeneous. By \eqref{eq:factorization} we then infer that the factors $p_{l,\polbul}$ are also $\gamma$-homogeneous. The homogeneity properties of discriminants and resultants (cf.\ \cite[\S A.1.3]{fischer_plane_2001}) allow us then to conclude that $\ddsnd$ is homogeneous; moreover, since the roots $R_{l,j}^\eta$ are uniquely determined by \eqref{eq:factorroots}, the $\gamma$-homogeneity of the $p_{l,\polbul}$ implies that the $\eta \mapsto R_{l,j}^\eta$ are homogeneous of degree $2$.
\end{proof}

\begin{lem}\label{lem:spectraldecomposition}
With the notation of Lemma~\ref{lem:factorization}, we can write
\begin{equation}\label{eq:spectraldecomposition}
-J_\eta^2 = \sum_{j=1}^\done (b_j^\eta)^2 P_j^\eta
\end{equation}
for all $\eta \in \ddsnd$, where the $P_j^\eta$ are orthogonal projections on $\fst$ of rank $2r_j$ for all $\eta \in \ddsnd$, with pairwise orthogonal ranges. In fact the $P_j^\eta$ are (componentwise) real analytic functions of $\eta \in \ddsnd$, homogeneous of degree $0$, and are rational functions of $\eta,b_1^\eta,\dots,b_\done^\eta$.
Moreover
\begin{equation}\label{eq:hsnorm}
\left( \sum_{\jone=1}^\done 2r_\jone (b_\jone^\eta)^2 \right)^{1/2} = (\tr(J_\eta^* J_\eta))^{1/2}
\end{equation}
and the last expression, as a function of $\eta$, is a norm induced by an inner product on $\dusnd$.
\end{lem}
\begin{proof}
For all $\eta \in \ddsnd$, \eqref{eq:spectraldecomposition} is the spectral decomposition of the selfadjoint endomorphism $-J_\eta^2$ given by the spectral theorem; the uniqueness of this decomposition, together with the homogeneity of $\eta \mapsto J_\eta$ and the $\eta \mapsto b_j^\eta$, implies the homogeneity of the $\eta \mapsto P_j^\eta$.

From the spectral decomposition \eqref{eq:spectraldecomposition} one deduces that $P_j^\eta = F_{j,\eta}(-J_\eta^2)$ for all (Borel) functions $F_{j,\eta} : \R \to \C$ such that $F_{j,\eta}(0) = 0$, $F_{j,\eta}((b_j^\eta)^2) = 1$, and $F_{j,\eta}((b_{j'}^\eta)^2) = 0$ for $j'\neq j$. If we choose as $F_{j,\eta}$ the interpolating polynomial
\[
F_{j,\eta}(\lambda) = \frac{\lambda \prod_{j'\neq j} (\lambda-(b_{j'}^\eta)^2)}{(b_j^\eta)^2 \prod_{j'\neq j} ((b_{j}^\eta)^2 - (b_{j'}^\eta)^2)},
\]
then it is clear that the $P_j^\eta$ are rational functions of $\eta,b_1^\eta,\dots,b_\done^\eta$ and that they are analytic on $\ddsnd$.

The identity \eqref{eq:hsnorm} is an immediate consequence of \eqref{eq:spectraldecomposition}. The right-hand side of \eqref{eq:hsnorm} is the pullback of the Hilbert-Schmidt norm via the map $\eta \mapsto J_\eta$, and since this map is injective (because $\snd = [\fst,\fst]$) the conclusion follows.
\end{proof}

From now on, let $\ddsnd$, $\done$, $r_0,r_1,\dots,r_\done$, $b_1^\eta,\dots,b_\done^\eta$, $P_1^\eta,\dots,P_\done^\eta$ be defined as in Lemmata~\ref{lem:factorization} and \ref{lem:spectraldecomposition}, and set $P_0^\eta = 1 - (P_1^\eta+\dots+P_\done^\eta)$.
Moreover, for all $n,k \in \N$, let
\[L_n^{(k)}(t) = \frac{t^{-k} e^t}{n!} \left( \frac{d}{dt} \right)^n (t^{k+n} e^{-t})\]
be the $n$-th Laguerre polynomial of type $k$, and define
\begin{align*}
\Ell_n^{(k)}(t) &= (-1)^n e^{-t} L_n^{(k)}(2t);
\end{align*}
for convenience, set $\Ell_n^{(k)} = 0$ for all  $n < 0$. In terms of these quantities, we can now write a formula for the convolution kernel of an operator in the joint functional calculus of $L,\vecU$. Namely, for all $H : \R \times \dusnd \to \C$, let $m_H : \R^\done \times \R \times \ddsnd \to \C$ be defined by
\begin{equation}\label{eq:multiplier}
m_H(n,\mu,\eta) = H\left(\sum_{j=1}^\done (2n_j+r_j) b_j^\eta + \mu, \eta\right).
\end{equation}

\begin{prp}\label{prp:kernel}
Suppose that $H : \R \times \dusnd \to \C$ is in the Schwartz class.
Then, for all $(x,u) \in G$,
\begin{equation}\label{eq:kernel}
\Kern_{H(L,\vecU)}(x,u) 
= \frac{2^{|r|}}{(2\pi)^{\dim G}}
 \int_{\ddsnd} \int_{\fst} V(\xi,\eta) \, e^{i \langle \xi, x \rangle} \, e^{i \langle \eta, u \rangle} \,d\xi \,d\eta,
\end{equation}
where $|r| = r_1 + \dots + r_\done$ and
\begin{equation}\label{eq:kernelfulltransform}
V(\xi,\eta) = \sum_{n \in \N^\done} m_H(n,|P^\eta_0 \xi|^2,\eta) 
\prod_{\jone=1}^\done \Ell_{n_\jone}^{(r_\jone-1)}(|P^\eta_\jone \xi|^2 /b^\eta_\jone).
\end{equation}
\end{prp}
\begin{proof}
Analogous to the proof of \cite[Proposition 4]{martini_heisenbergreiter}.
\end{proof}

The following identities are easily obtained from the properties of Laguerre polynomials (see, e.g., \cite[\S10.12]{erdelyi_higher2_1981}).

\begin{lem}\label{lem:laguerre}
For all $k,n,n' \in \N$ and $t \in \R$,
\begin{gather}
\label{eq:laguerred} \frac{d}{dt} \Ell_n^{(k)}(t) = \Ell_{n-1}^{(k+1)}(t) - \Ell_{n}^{(k+1)}(t),\\
\label{eq:laguerreo} \int_{0}^{\infty} \Ell_n^{(k)}(t) \, \Ell_{n'}^{(k)}(t) \, t^k \,dt = \begin{cases}
\frac{(n+k)!}{2^{k+1} n!} &\text{if $n=n'$,}\\
0 &\text{otherwise.}
\end{cases}
\end{gather}
\end{lem}

Proposition~\ref{prp:kernel}, together with the Plancherel formula for the Euclidean Fourier transform and the orthogonality properties \eqref{eq:laguerreo} of the Laguerre functions, allows us to compute the Plancherel measure associated to the system \eqref{eq:operators} of commuting operators in the sense of \cite{martini_spectral_2011,martini_joint_2012}.

\begin{cor}\label{cor:plancherelmeasure}
For all $H : \R \times \dusnd \to \C$ in the Schwartz class, if $m$ is defined as in \eqref{eq:multiplier}, then
\begin{multline*}
\int_G |\Kern_{H(L,\vecU)}(x,u)|^2 \,dx \,du \\
= (2\pi)^{|r|-\dim G} \int_{\ddsnd} \int_{\leftclosedint 0,\infty \rightopenint} \sum_{n \in \N^\done} |m_H(n,\mu,\eta)|^2 \, \prod_{\jone=1}^\done \Bigl[ (b_\jone^\eta)^{r_\jone} {\textstyle \binom{n_\jone+r_\jone-1}{n_\jone}} \Bigr] \,d\sigma_{r_0}(\mu) \,d\eta,
\end{multline*}
where $|r| = r_1 + \dots + r_\done$, $\sigma_{r_0}$ is the Dirac delta at $0$ if $r_0 = 0$, and
\[
d\sigma_{r_0}(\mu) = \frac{\pi^{r_0/2}}{\Gamma(r_0/2)} \mu^{r_0/2} \,\frac{d\mu}{\mu}
\]
if $r_0 > 0$.
\end{cor}

\section{Self-controlled functions and differential polynomials}\label{section:selfcontrol}

By \eqref{eq:kernel} and integration by parts, the multiplication of $\Kern_{H(L,\vecU)}(x,u)$ by polynomial functions of $u$ corresponds to taking $\eta$-derivatives of $V(\xi,\eta)$ in \eqref{eq:kernelfulltransform}; we are then interested in estimating $\eta$-derivatives of $V(\xi,\eta)$ in terms of derivatives of the multiplier $H$, or rather of its reparametrization $m_H$. The expressions for these derivatives obtained from \eqref{eq:kernelfulltransform} can be quite complicated; nevertheless we will show that they have a specific form, which is ``self-reproducing'', so that they can can be estimated (under suitable assumptions on $\eta$-derivatives of $b_1^\eta,\dots,b_\done^\eta,P_1^\eta,\dots,P_\done^\eta$)  by (finite sums of) expressions analogous to \eqref{eq:kernelfulltransform}, where $m_H$ is replaced by some derivative of $m_H$. In order to give a precise meaning to these ideas, in this section we introduce some definitions and notation, which will be then exploited in the following \S\ref{section:kernelderivatives} to deal with derivatives of $V(\xi,\eta)$.

Let $\Omega$ be a smooth manifold. Let $D = (D_1,\dots,D_n)$ be a system of smooth commuting vector fields on $\Omega$. Set $D^\alpha = D_1^{\alpha_1} \cdots D_n^{\alpha_n}$ and $|\alpha| = \alpha_1 + \dots + \alpha_n$ for all multiindices $\alpha \in \N^n$. Inequalities between multiindices shall be interpreted componentwise.

For all $k \in \N$ and all functions $g : \Omega \to \C$, let $\BD_{\Omega,D}^k(g)$ be the set of the smooth functions $f : \Omega \to \C$ such that there exists a constant $C \geq 0$ such that for all $\alpha \in \N^n$ with $|\alpha| \leq k$,
\[|D^\alpha f| \leq C |g|\]
($\BD$ stands for ``bounded derivatives''); the minimum of these constants $C$ will be denoted as $\|f\|_{\BD_{\Omega,D}^k(g)}$. We denote moreover by $\BD_{\Omega,D}^\infty(g)$ the intersection $\bigcap_{k \in \N} \BD_{\Omega,D}^k(g)$. When $f$ is a smooth $\R^m$-valued function on $\Omega$, we will write $f \in \BD_{\Omega,D}^k(g)$ to express that all the components of $f$ belong to $\BD_{\Omega,D}^k(g)$.

In the following we will have to deal with expressions given by linear combinations of products of iterated derivatives $D^\alpha f$ of a given function $f$. Since we need to keep track of the form of these expressions, independently of the choice of $f$ or $D$, it is convenient to introduce the following definition.
Fix a system $(X_\alpha)_{\alpha \in \N^n}$ of indeterminates (one should think of each indeterminate $X_\alpha$ as representing an iterated derivative $D^\alpha f$).
For all $k \in \N \cup \{\infty\}$ and $r \in \N$, let $\HDP_n^k(r)$ be the set of homogeneous polynomials of degree $r$ with complex coefficients and indeterminates from $(X_\alpha)_{\alpha \in \N^n,|\alpha|\leq k}$ ($\HDP$ stands for ``homogeneous differential polynomial''). 
If $p \in \HDP_n^\infty(r)$ and $\gamma \in \N^n$, we denote by $\partial^\gamma p$ the polynomial given by
\[
\partial^\gamma p \defeq \sum_{\alpha} X_{\alpha+\gamma} \frac{\partial p}{\partial X_\alpha}
\]
(note that only a finite number of summands is nonzero). Further, if $p \in \HDP_n^\infty(r)$ and $f \in C^\infty(\Omega)$, we denote by $p(D;f)$ the function obtained from $p$ by replacing the indeterminate $X_\alpha$ with $D^\alpha f$ for all $\alpha \in \N^n$. The basic properties of the classes $\HDP$ are summarized in the following lemma.

\begin{lem}\label{lem:hdp}
Let $f,h : \Omega \to \C$ be smooth, $\kappa \in \leftclosedint 0,\infty \rightopenint$, $k,r \in \N$.
\begin{itemize}
\item[(i)] If $p \in \HDP^k_n(r)$, then $\partial^\gamma p \in \HDP^{k+|\gamma|}_n(r)$ and $(\partial^\gamma p)(D;f) = D^\gamma (p(D;f))$.
\item[(ii)] If $p \in \HDP_{n}^k(r)$ and $\|f\|_{\BD_{\Omega,D}^k(h)} \leq \kappa$, then $|p(D;f)| \leq C_{p,\kappa} |h|^r$.
\end{itemize}
\end{lem}
\begin{proof}
Part (i) follows from Leibniz' rule, while part (ii) is immediate from the definitions.
\end{proof}

For all $k \in \N \cup \{\infty\}$, let $\SC_{\Omega,D}^k$ be the set of the smooth functions $f : \Omega \to \C$ such that $f \in \BD_{\Omega,D}^k(f)$ ($\SC$ stands for ``self-controlled''). When $k < \infty$, we set $\|f\|_{\SC^k_{\Omega,D}} = \|f\|_{\BD^k_{\Omega,D}(f)}$. Note that
\begin{equation}\label{eq:sc_scalar}
\| \lambda f \|_{\SC^k_{\Omega,D}} = \| f \|_{\SC^k_{\Omega,D}}
\end{equation}
for all $\lambda \in \C \setminus \{0\}$. We now show some closure properties of the classes $\SC$.

\begin{lem}\label{lem:sc}
Let $k \in \N \cup \{\infty\}$.
\begin{itemize}
\item[(i)] The constant functions belong to $\SC_{\Omega,D}^k$.
\item[(ii)] If $f,g \in \SC_{\Omega,D}^k$ and $f,g \geq 0$, then $f+g \in \SC_{\Omega,D}^k$.
\item[(iii)] If $f,g \in \SC_{\Omega,D}^k$, then $fg \in \SC_{\Omega,D}^k$.
\item[(iv)] If $f \in \SC_{\Omega,D}^k$, $|f| > 0$, and $r \in \Z$, then $f^r \in \SC_{\Omega,D}^k$.
\item[(v)] If $f \in \SC_{\Omega,D}^k$, $f > 0$, and $r \in \C$, then $f^r \in \SC_{\Omega,D}^k$.
\end{itemize}
If moreover $k < \infty$ and $\|f\|_{\SC_{\Omega,D}^k},\|g\|_{\SC_{\Omega,D}^k} \leq \kappa$ for some $\kappa \in \leftclosedint 0,\infty\rightopenint$, then $\|f+g\|_{\SC_{\Omega,D}^k} \leq \kappa$, $\|fg\|_{\SC_{\Omega,D}^k} \leq C_{n,k,\kappa}$, $\|f^r\|_{\SC_{\Omega,D}^k} \leq C_{n,k,r,\kappa}$ in the cases (ii), (iii), (iv) and (v) respectively.
\end{lem}
\begin{proof}
Part (i) is trivial. Part (ii) follows from the linearity of $D^\alpha$ and the fact that $|f+g| = |f|+|g|$ when $f,g \geq 0$. Part (iii) follows from Leibniz' rule. As for part (iv) and part (v), from the identity $D_j (f^r) = r f^{r-1} D_j f$ one deduces inductively via Leibniz' rule and Lemma~\ref{lem:hdp}(i) that $D^\alpha (f^r) = f^{r-|\alpha|} \Phi_{r,\alpha}(D;f)$, where $\Phi_{r,\alpha} \in \HDP_{n}^{|\alpha|}(|\alpha|)$, and consequently $|D^\alpha (f^r)| \leq C_{r,\alpha,\kappa} |f^r|$ whenever $|\alpha| \leq k$ by Lemma~\ref{lem:hdp}(ii).
\end{proof}

The following lemma deals with the behavior of the class $\SC$ under composition; it will be particularly useful in proving uniform estimates for cutoff functions.

\begin{lem}\label{lem:sccomp2}
Let $k \in \N$ and $\kappa \in \leftclosedint 0,\infty \rightopenint$. Let $I \subseteq \R$ be open, $f : \Omega \to I$ and $g : I \to \C$ be smooth. Suppose that $\|f\|_{\SC_{\Omega,D}^k}, \|g\|_{C^k(f(\Omega))} \leq \kappa$, $f(\Omega) \cap \supp g \subseteq \leftclosedint -\kappa,\kappa\rightclosedint$. Then:
\begin{itemize}
\item[(i)] $g \circ f \in \BD_{\Omega,D}^k(1)$ and $\|g\circ f\|_{\BD_{\Omega,D}^k(1)} \leq C_{n,k,\kappa}$;
\item[(ii)] if moreover $|g(x)| \geq \kappa^{-1} |x|$ for all $x \in f(\Omega)$, then $g \circ f \in \SC^k_{\Omega,D}$ and $\|g \circ f\|_{\SC_{\Omega,D}^k} \leq C_{n,k,\kappa}$.
\end{itemize}
\end{lem}
\begin{proof}
Let $\alpha \in \N^N$ be such that $|\alpha| \leq k$. If $\alpha=0$, then it is obvious that $|D^\alpha(g\circ f)|$ is bounded by $\kappa$ and also by $|g\circ f|$. Suppose instead that $\alpha\neq 0$. Iterated application of Leibniz' rule and Lemma~\ref{lem:hdp}(i) gives that
\[
D^\alpha (g \circ f) = \sum_{h=1}^{|\alpha|} (g^{(h)} \circ f) \cdot \Psi_{\alpha,h}(D;f),
\]
where $\Psi_{\alpha,h} \in \HDP^{|\alpha|}_{n}(h)$. Since the $|g^{(h)}|$ are bounded by $\kappa$ on $U$, from Lemma~\ref{lem:hdp}(ii) we obtain
\[
|D^\alpha (g \circ f)| \leq C_{\alpha,\kappa} \sum_{h=1}^{|\alpha|} (\tilde g \circ f) \, |f|^h,
\]
where $\tilde g$ is the characteristic function of $\supp g$. Since $\supp g \subseteq \leftclosedint -\kappa,\kappa\rightclosedint$, from this inequality we deduce immediately that $|D^\alpha (g \circ f)| \leq C_{\alpha,\kappa}$, hence part (i) follows. Since $h \geq 1$ in the sum above, the same inequality yields also $|D^\alpha (g \circ f)| \leq C_{\alpha,\kappa} |f|$; in the case $|g(x)| \geq \kappa^{-1} |x|$ for all $x \in f(\Omega)$, we have $|f| \leq \kappa |g\circ f|$, and part (ii) follows.
\end{proof}

Let us now specialize to the case where $\Omega$ is an open subset of $\R^n$, with coordinates $(\eta_1,\dots,\eta_n)$, and $D = (\eta_1\partial_{\eta_1},\dots,\eta_n \partial_{\eta_n})$. In this case, homogeneity properties (together with the compactness of the unit sphere $S^{n-1}$ of $\R^n$) can be used to show that a function belongs to some classes $\BD$ or $\SC$. In fact, one can obtain estimates independent of the choice of (orthonormal) coordinates.

\begin{lem}\label{lem:schom_unif}
Let $\Omega \subseteq \R^n \setminus \{0\}$ be open and homogeneous. Let $f : \overline{\Omega} \setminus \{0\} \to \C$ be homogeneous of degree $r \in \C$ and admitting a smooth extension to some open neighborhood of $\overline{\Omega} \setminus \{0\}$ in $\R^n \setminus \{0\}$. Then, for all $k \in \N$, there exists $C_{f,k} \in \leftclosedint 0,\infty \rightopenint$ such that, for all choices of orthonormal coordinates $(\tilde\eta_1,\dots,\tilde\eta_n)$ on $\R^n$, if $D = (\tilde\eta_1 \partial_{\tilde\eta_1}, \dots, \tilde\eta_n \partial_{\tilde\eta_n})$, then
\begin{itemize}
\item[(i)] $\|f\|_{\BD_{\Omega,D}^k(\eta \mapsto |\eta|^{\Re r})} \leq C_{f,k}$ and,
\item[(ii)] if moreover $f$ does not vanish in $\overline{\Omega} \setminus \{0\}$, then $\|f\|_{\SC_{\Omega,D}^k} \leq C_{f,k}$.
\end{itemize}
\end{lem}
\begin{proof}
Denote by $\nabla^k f$ the symmetric $k$-tensor of $k$th-order derivatives of $f$. If $(\tilde e_1,\dots,\tilde e_n)$ is the orthonormal basis of $\R^n$ associated to the coordinates $(\tilde\eta_1,\dots,\tilde\eta_n)$, then $\partial_{\tilde\eta}^\alpha f = \langle \nabla^{|\alpha|} f, \tilde e_1^{\otimes \alpha_1} \otimes \dots \otimes \tilde e_n^{\otimes \alpha_n} \rangle$, and consequently
\[
|D^\alpha f(\eta)| = \left|\sum_{\beta \leq \alpha} C_{\alpha,\beta} \, \tilde\eta^\beta \partial_{\tilde\eta}^\beta f(\eta) \right| \leq C_{\alpha} \max_{k \leq |\alpha|} |\eta|^k |\nabla^k f(\eta)|.
\]
From this inequality, the continuity of $\nabla^k f$ and the compactness of $S^{n-1} \cap \overline{\Omega}$, we deduce that $|D^\alpha f|$ can be majorized on $S^{n-1} \cap \overline{\Omega}$ by a constant $C_{f,\alpha}$ not depending on the choice of coordinates; since $D^\alpha f$ is homogeneous of degree $r$, we then deduce that
\[
|D^\alpha f(\eta)| \leq C_{f,\alpha} |\eta|^{\Re r}
\]
for all $\eta \in \overline{\Omega} \setminus \{0\}$, and part (i) is proved. On the other hand, if $f$ does not vanish on $\overline{\Omega} \setminus \{0\}$, then by compactness and homogeneity we deduce that
\[
|\eta|^{\Re r} \leq C_f |f(\eta)|
\]
for all $\eta \in \overline{\Omega} \setminus \{0\}$, and part (ii) follows by combining the two inequalities.
\end{proof}

A multivariate analogue of the previous argument, exploiting the compactness of the product of unit spheres, yields immediately the following result.

\begin{lem}\label{lem:schom_multi}
Let $\Omega = (\R^{n_1} \setminus \{0\}) \times \dots \times (\R^{n_s} \setminus \{0\})$. Suppose that $f : \Omega \to \C$ is smooth and multihomogeneous of degree $r \in \C^s$, i.e.,
\[
f(\lambda_1 \eta_1,\dots,\lambda_s \eta_s) = \lambda_1^{r_1} \dots \lambda_s^{r_s} f(\eta_1,\dots,\eta_s),
\]
for all $\lambda_1,\dots,\lambda_s \in \leftopenint 0,\infty \rightopenint$ and $\eta=(\eta_1,\dots,\eta_s) \in \Omega$. Then, for all $k \in \N$, there exists $C_{f,k} \in \leftclosedint 0,\infty \rightopenint$ such that, for all choices of orthonormal coordinates $(\tilde\eta_{l,1},\dots,\tilde\eta_{l,n_l})$ on $\R^{n_l}$ for $l=1,\dots,s$, if
\[
D = (\tilde\eta_{1,1} \partial_{\tilde\eta_{1,1}}, \dots, \tilde\eta_{1,n_1} \partial_{\tilde\eta_{1,n_1}}, \dots, \tilde\eta_{s,1} \partial_{\tilde\eta_{s,1}}, \dots, \tilde\eta_{s,n_s} \partial_{\tilde\eta_{s,n_s}}),
\]
then
\begin{itemize}
\item[(i)] $\|f\|_{\BD_{\Omega,D}^k(\eta \mapsto |\eta_1|^{\Re r_1} \cdots |\eta_s|^{\Re r_s})} \leq C_{f,k}$ and,
\item[(ii)] if moreover $f$ does not vanish in $\Omega$, then $\|f\|_{\SC_{\Omega,D}^k} \leq C_{f,k}$.
\end{itemize}
\end{lem}

We conclude this section by briefly recalling the construction of smooth homogeneous partitions of unity on $\R^n \setminus \{0\}$ (i.e., partitions of unity made of smooth functions homogeneous of degree $0$) depending on a thinness parameter $\thinp$ (i.e., corresponding to the choice of an $\epsilon$-separated set $I_\epsilon$ of unit vectors), which have been extensively used in the literature (see, e.g., \cite{fefferman_note_1973,christ_weak_1988,seeger_regularity_1991}), and will be useful in \S\ref{section:g743} below. The language introduced above can be used to express the uniformity in $\thinp$ of the estimates on the derivatives of the components of the partitions of unity.

\begin{lem}\label{lem:hompartuni}
For all $\thinp \in \leftopenint 0,1 \rightclosedint$, there exist a finite subset $I_\thinp$ of $S^{n-1}$ and a smooth homogeneous partition of unity $(\chi_{\thinp,v})_{v \in I_\thinp}$ on $\R^n \setminus \{0\}$ such that
\begin{itemize}
\item[(i)] the cardinality of $I_\thinp$ is at most $C \thinp^{1-n}$
\end{itemize}
and, for all $v \in I_\thinp$,
\begin{itemize}
\item[(ii)] $\supp \chi_{\thinp,v} \subseteq \{ \xi \tc \thinp/4 \leq | \xi/|\xi| - v | \leq 4\thinp \}$;
\item[(iii)] if $(\xi_1^v,\dots,\xi_n^v)$ are orthonormal coordinates on $\R^n$ such that $v$ corresponds to $(1,0,\dots,0)$, then, for all $\alpha \in \N^n$ and $\xi \in \R^n \setminus \{0\}$,
\[
|\partial_{\xi^v}^\alpha \chi_{\thinp,v}(\xi)| \leq C_\alpha |\xi|^{-|\alpha|} \thinp^{\alpha_1-|\alpha|};
\]
\item[(iv)] if moreover $D_v = (\xi^v_1 \partial_{\xi^v_1}, \dots, \xi^v_n \partial_{\xi^v_n})$, then, for all $k \in \N$,
\[
\| \chi_{\thinp,v} \|_{\BD^k_{\R^n \setminus \{0\},D_v}(1)} \leq C_k.
\]
\end{itemize}
\end{lem}

Note that the above constants $C,C_\alpha,C_k$ do not depend on $\thinp$ or $v$, but may depend on the dimension $n$. Note further that, differently from the standard construction, here we require $\chi_{\thinp,v}$ not only to be supported in a conic neighborhood of the direction $v$, but also to vanish on a smaller conic neighborhood of $v$; this property will be convenient to estimate from above and from below the size of the ``transversal component'' $(\xi_2^v,\dots,\xi_n^v)$ of a point $\xi$ in the support of $\chi_{\thinp,v}$.

\begin{proof}
We follow, with slight variations, the construction given in \cite[\S IX.4]{stein_harmonic_1993}. Let $I_\thinp$ be a subset of $S^{n-1}$ such that
\[
|v-v'| \geq \thinp \text{ for all $v,v' \in I_\thinp$ with $v\neq v'$}
\]
and maximal among the subsets of $S^{n-1}$ with this property. A moment's reflection shows that (i) is satisfied, that
\[
\text{for all $v \in I_\epsilon$ there exists $v'\in I_\epsilon$ such that } \epsilon \leq |v-v'| \leq 2\epsilon,
\]
and that
\begin{equation}\label{eq:reticolo}
\text{for all } \xi \in S^{n-1} \text{ there is } v \in I_\thinp \text{ such that } (1/2) \, \thinp \leq |v-\xi| < (5/2) \, \thinp.
\end{equation}

Choose a smooth $\phi : \R^n \to \R$ such that $\phi(\xi) = 1$ for $1/2 \leq |\xi| \leq 5/2$ and $\supp \phi \subseteq \{ \xi \tc 1/4 \leq |\xi| \leq 4 \}$, and set
\[
\tilde\chi_{\thinp,v}(\xi) = \phi(\thinp^{-1}(\xi/|\xi|-v)), \qquad \chi_{\thinp,v} = \frac{\tilde\chi_{\thinp,v}}{\sum_{v'\in I_\thinp} \tilde\chi_{\thinp,v'}}.
\]
By \eqref{eq:reticolo} $\chi_{\thinp,v}$ is well-defined and smooth on $\R^n \setminus \{0\}$, and clearly it is homogeneous of degree $0$ and satisfies (ii); further
\[
|\partial^\alpha_{\tilde\xi} \chi_{\thinp,v}(\xi)| \leq C_\alpha |\xi|^{-|\alpha|} \thinp^{-|\alpha|}
\]
for all $\alpha \in \N^n$ and all choices of orthonormal coordinates $(\tilde\xi_1,\dots,\tilde\xi_n)$ on $\R^n$ (cf.\ \cite[\S IX.4.4]{stein_harmonic_1993}).

The last inequality proves (iii) in the case $\alpha_1 = 0$; for $\alpha_1 > 0$, one then proceeds by induction on $\alpha_1$, by exploiting the identity
\[
(\xi^v_1)^k \partial_{\xi_1^v}^k = \varrho^k \partial_\varrho^k - \sum_{|\beta| = k, \, \beta_1 < k} c_\beta \, (\xi^v)^\beta \partial_{\xi^v}^\beta
\]
(where $\varrho = |\xi|$, $\partial_\varrho$ is the radial derivative, and $c_\beta \in \N$), the fact that
\[
\varrho^k \partial_\varrho^k f = c_{k,\gamma} f \qquad\text{for all $f$ homogeneous of degree $\gamma$}
\]
(where $c_{k,\gamma} = \gamma (\gamma-1) \dots (\gamma-k+1)$) and that
\begin{equation}\label{eq:stimesupporto}
\xi^v_1 \sim |\xi|, \quad |\xi^v_2|,\dots,|\xi^v_n| \lesssim \thinp |\xi| \qquad\text{for all $\xi \in \supp \chi_{\thinp,v}$}.
\end{equation}

From (iii) and \eqref{eq:stimesupporto} we deduce in particular that
\[
|(\xi^v)^\alpha \partial_{\xi_v}^\alpha \chi_{\thinp,v}(\xi)| \leq C_\alpha,
\]
and (iv) follows because $D_v^\alpha = \sum_{\beta \leq \alpha} c_{\alpha,\beta} (\xi^v)^\beta \partial_{\xi^v}^\beta$ for some $c_{\alpha,\beta} \in \N$.
\end{proof}

\section{Derivatives of the kernel formula}\label{section:kernelderivatives}

By using the notation of \S\ref{section:selfcontrol}, we can now show that iterated $\eta$-derivatives of $V$ in \eqref{eq:kernelfulltransform} have a precise form.

Let $e_1,\dots,e_\done$ denote the standard basis of $\R^\done$, and let $D = (D_1,\dots,D_N)$ be a commuting system of smooth vector fields on $\ddsnd$. We introduce some operators on functions $f : \R^\done \times \R \times \ddsnd \to \C$ as follows:
\begin{align*}
\delta_j f(n,\mu,\eta) &= f(n+e_j,\mu,\eta) - f(n,\mu,\eta)\\
\partial_{n_j} f(n,\mu,\eta) &= \frac{\partial}{\partial n_j} f(n,\mu,\eta),\\
\partial_\mu f(n,\mu,\eta) &= \frac{\partial}{\partial\mu} f(n,\mu,\eta),\\
D_k f(n,\mu,\eta) &= D_k (f(n,\mu,\cdot))(\eta)
\end{align*}
for all $j \in \{1,\dots,\done\}$, $k \in \{1,\dots,N\}$. As usual, for all $\alpha \in \N^N$ and $\beta \in \N^\done$, we set $D^\alpha = D_1^{\alpha_1} \dots D_N^{\alpha_N}$, $\delta^\beta = \delta_1^{\beta_1} \dots \delta_\done^{\beta_\done}$, $\partial_n^\beta = \partial_{n_1}^{\beta_1} \dots \partial_{n_\done}^{\beta_\done}$.

\begin{prp}\label{prp:derivatives}
Let $D = (D_1,\dots,D_N)$ be a commuting system of smooth vector fields on $\ddsnd$.
For all $H : \R \times \dusnd \to \C$ smooth and compactly supported in $\R \times \ddsnd$, if $V$ is defined as in \eqref{eq:kernelfulltransform}, then, for all $\alpha \in \N^N$, $\eta \in \ddsnd$, $\xi \in \fst$,
\begin{multline}\label{eq:derivatives}
D^\alpha V(\xi,\eta) 
= \sum_{\iota \in I_{\alpha}} \sum_{n \in \N^\done} D^{\gamma^\iota} \partial_\mu^{k_\iota} \delta^{\beta^\iota} m_H(n,|P_0^\eta \xi|^2,\eta) \, \Psi_{\iota,0}(D;|P_0^\eta \xi|^2) \\
\times \prod_{\jone=1}^\done \left[ \Ell^{(r_\jone-1+\beta^\iota_\jone)}_{n_\jone}(|P_\jone^\eta \xi|^2/b_\jone^\eta) \, \Phi_{\iota,\jone}(D;1/b_\jone^\eta) \, \Psi_{\iota,\jone}(D;|P_\jone^\eta \xi|^2) \right] 
\end{multline}
where $I_{\alpha}$ is a finite set and, for all $\iota \in I_{\alpha}$,
\begin{itemize}
\item $\gamma^\iota \in \N^N$, $k_\iota \in \N$, $\beta^\iota \in \N^\done$, $\gamma^\iota \leq \alpha$,
\item $\Psi_{\iota,0} \in \HDP^{|\alpha|}_{N}(k_\iota)$,
\item for $\jone=1,\dots,\done$, $\Phi_{\iota,\jone} \in \HDP^{|\alpha|}_{N}(\beta^\iota_\jone)$,
\item for $\jone=1,\dots,\done$, $\Psi_{\iota,\jone} = \Psi_{\iota,\jone}^0 \Psi_{\iota,\jone}^1$, where $\Psi^1_{\iota,\jone} \in \HDP^{|\alpha|}_{N}(q^\iota_\jone)$ and $\Psi_{\iota,\jone}^0 \in \HDP^{1}_{N}(\beta^\iota_\jone - q^\iota_\jone)$, for some $q^\iota \in \N^\done$ such that $q^\iota \leq \beta^\iota$,
\item $\min\{1,|\alpha|\} \leq |\beta^\iota| + k_\iota + |\gamma^\iota|$ and $|\beta^\iota| + k_\iota + |\gamma^\iota|+ |q^\iota| \leq |\alpha|$.
\end{itemize}
\end{prp}
\begin{proof}
Notice first that the above statement can be equivalently rephrased by additionally requiring that each of the polynomials $\Psi_{\iota,j}^0$ is made of a unique monic monomial (it is sufficient to rearrange the sum). Hence we may suppose that $\Psi_{\iota,j}^0(D;|P_\jone^\eta \xi|^2)$ is just a product of factors of the form $|P_\jone^\eta \xi|^2$ or $D_\jtwo |P_\jone^\eta\xi|^2$ for $\jtwo \in \{1,\dots,N\}$.

The proof goes by induction on $|\alpha|$. The case $\alpha=0$ is given by Proposition~\ref{prp:kernel}. For the inductive step, one employs Leibniz' rule, and the following observations:
\begin{itemize}
\item when a $D_\jtwo$-derivative hits $D^{\gamma^\iota} \partial_\mu^{k_\iota} \delta^{\beta^\iota} m_H(n,|P_0^\eta \xi|^2,\eta)$, either it increases $\gamma^\iota_\jtwo$, or it increases $k_\iota$; in the second case, the degree of $\Psi_{\iota,0}$ is increased too;
\item when a $D_\jtwo$-derivative hits $\Ell^{(r_\jone-1+\beta^\iota_\jone)}_{n_\jone}(|P_\jone^\eta \xi|^2/b_\jone^\eta)$, then \eqref{eq:laguerred} and summation by parts in $n_\jone$ increase the order $\beta^\iota_\jone$ of discrete differentiation in $\delta_\jone$, and moreover the additional factor
\[
D_\jtwo(|P^\eta_\jone \xi|^2/b_\jone^\eta) = (1/b_\jone^\eta) D_\jtwo(|P^\eta_\jone \xi|^2) + |P^\eta_\jone \xi|^2 D_\jtwo(1/b_\jone^\eta)
\]
given by the chain rule increases the degrees of both $\Phi_{\iota,\jone}$ and $\Psi^0_{\iota,\jone}$;
\item when a $D_\jtwo$-derivative hits $\Psi_{\iota,0}(D;|P_0^\eta\xi|^2)$, $\Phi_{\iota,\jone}(D;1/b_\jone^\eta)$ or $\Psi_{\iota,\jone}^1(D;|P_\jone^\eta\xi|^2)$, ``nothing happens'' because of Lemma \ref{lem:hdp}(i);
\item when a $D_\jtwo$-derivative hits a factor of $\Psi_{\iota,\jone}^0(D;|P_\jone^\eta\xi|^2)$, the derivative of this factor can be included in the new $\Psi^1_{\iota,\jone}(D;|P_\jone^\eta\xi|^2)$, hence the degree $q^\iota_\jone$ of $\Psi^1_{\iota,\jone}$ increases, while the degree of $\Psi_{\iota,\jone}^0$ decreases, and the sum $\beta^\iota_\jone$ of the degrees is unvaried.
\end{itemize}
The conclusion follows.
\end{proof}

The previous formula expresses $D^\alpha V$ in terms of derivatives $D^{\gamma^\iota} \partial_\mu^{k_\iota} \delta^{\beta^\iota} m_H$ of the reparametrized multiplier $m_H$ whose total order $|\beta^\iota| + k_\iota + |\gamma^\iota|$ does not exceed $|\alpha|$. We will now convert this formula into an $L^2$-estimate, by exploiting the orthogonality properties of the Laguerre functions. More precisely, we will use an enhanced version of the orthogonality relations \eqref{eq:laguerreo}, allowing for a mismatch between the type of Laguerre functions and the exponent in the weight defining the measure. As we will see, this mismatch may produce additional discrete differentiations of $m_H$;
nevertheless, the total order of differentiation will not exceed $|\alpha|$, thanks to the fact that in Proposition~\ref{prp:derivatives} the degrees $q_\jone^\iota$ of the $\Psi_{\iota,\jone}^1$ are also kept under control.

Note that, for all $f : \R^\done \times \R \times \ddsnd \to \C$, $\alpha \in \N^\done$, $\mu \in \R$ and $\eta \in \ddsnd$, the functions $\delta^\alpha f(\cdot,\mu,\eta)$ and $\partial_n^\alpha f(\cdot,\mu,\eta)$ depend only on $f(\cdot,\mu,\eta)$; in other words, $\delta^\alpha$ and $\partial_n^\alpha$ can be thought of as operators on functions $\R^\done \to \C$. Set moreover $\langle s \rangle = 1 + |s|$ and $(s)_+ = \max\{s,0\}$ for all $s \in \R$. For a multiindex $\alpha = (\alpha_1,\dots,\alpha_\done) \in \R^\done$, set $(\alpha)_+ = ((\alpha_1)_+,\dots,(\alpha_\done)_+)$. The aforementioned ``enhanced orthogonality relations'' can be then stated as follows.

\begin{lem}\label{lem:laguerreorthogonality}
For all $h,k \in \N^\done$ and all compactly supported $f : \R^\done \to \C$,
\begin{multline*}
\int_{\leftopenint 0,\infty\rightopenint^\done} \Bigl| \sum_{n \in \N^\done} f(n) \, \prod_{j=1}^\done \Ell_{n_j}^{(k_j)}(t_j) \Bigr|^2 \,t^h \,dt \\
\leq C_{h,k} \sum_{n \in \N^\done} |\delta^{(k-h)_+} f(n)|^2 \, \prod_{j=1}^\done \langle n_j\rangle^{h_j+2(k_j-h_j)_+}.
\end{multline*}
\end{lem}
\begin{proof}
See \cite[Lemma~7]{martini_heisenbergreiter}.
\end{proof}

Another simple remark will be of use: via the fundamental theorem of integral calculus, finite differences can be estimated by continuous derivatives.

\begin{lem}\label{lem:discretecontinuous}
Let $f : \R^\done \to \C$ be smooth, and let $\beta \in \N^\done$. Then
\[\delta^\beta f(n) = \int_{J_\beta} \partial_n^\beta f(n+s) \,d\nu_\beta(s)\]
for all $n \in \R^\done$, where $J_\beta = \prod_{\jone=1}^\done \leftclosedint 0,\beta_\jone\rightclosedint$ and $\nu_\beta$ is a Borel probability measure on $J_\beta$. In particular
\[|\delta^\beta f(n)|^2 \leq \int_{J_\beta} |\partial^\beta f(n+s)|^2 \,d\nu_\beta(s)\]
for all $n \in \R^\done$.
\end{lem}

We now have all the ingredients to obtain from Proposition~\ref{prp:derivatives}, under suitable assumptions on $\eta$-derivatives of $b_1^\eta,\dots,b_\done^\eta,P_1^\eta,\dots,P_\done^\eta$, an estimate for a partial $L^2$-norm of $D^\alpha V(\xi,\eta)$ in terms of derivatives of the reparametrized multiplier $m_H$ of order at most $|\alpha|$. A comparison of this estimate with Corollary~\ref{cor:plancherelmeasure} shows the ``self-reproducing'' character of the formulas under consideration.

\begin{prp}\label{prp:partiall2norm}
Let $D = (D_1,\dots,D_N)$ be a commuting system of smooth vector fields on $\ddsnd$. Let $\Omega \subseteq \ddsnd$ be open, and suppose that
\begin{equation}\label{eq:selfcontrolhypothesis}
\|b_1^\eta\|_{\SC^A_{\Omega,D}},\dots,\|b_\done^\eta\|_{\SC^A_{\Omega,D}}, \|P_1^\eta\|_{\BD_{\Omega,D}^A(1)},\dots,\|P_\done^\eta\|_{\BD_{\Omega,D}^A(1)} \leq \kappa
\end{equation}
for some $A \in \N$ and $\kappa \in \leftclosedint 0,\infty \rightopenint$. For all $H : \R \times \dusnd \to \C$ smooth and compactly supported in $\R \times \Omega$, and for all $\alpha \in \N^N$ with $|\alpha| \leq A$, if $V$ is defined as in \eqref{eq:kernelfulltransform}, then
\begin{multline*}
\int_{\fst} \Bigl|D^\alpha V(\xi,\eta) \Bigr|^2 \, d\xi 
\leq C_{\kappa,\alpha} \sum_{\iota \in I'_{\alpha}} \int_{\leftclosedint 0,\infty \rightopenint} \int_{J_\iota} \sum_{n \in \N^\done} |D^{\gamma^\iota} \partial_\mu^{k_\iota} \partial_n^{\beta^\iota} m_H(n+s,\mu,\eta)|^2 \\
\times \prod_{\jone=1}^\done \left[ (b^\eta_\jone)^{1+a^\iota_\jone-2\beta^\iota_\jone} \langle n_\jone \rangle^{a^\iota_\jone} \right] \,d\nu_{\iota}(s) \,d\sigma_\iota(\mu),
\end{multline*}
for all $\eta \in \ddsnd$, where $I'_{\alpha}$ is a finite set and, for all $\iota \in I'_{\alpha}$,
\begin{itemize}
\item $\beta^\iota,a^\iota \in \N^\done$, $k_\iota \in \N$, $\gamma^\iota \in \N^N$,
\item $\gamma^\iota \leq \alpha$ and $a^\iota_\jone \geq r_\jone-1$ for $\jone=1,\dots,\done$,
\item $\min\{1,|\alpha|\} \leq |\gamma^\iota| + k_\iota + |\beta^\iota| \leq |\alpha|$,
\item $J_\iota = \prod_{\jone=1}^\done \leftclosedint 0,\beta^\iota_\jone \rightclosedint$ and $\nu_\iota$ is a Borel probability measure on $J_\iota$,
\item $\sigma_\iota$ is a regular Borel measure on $\leftclosedint 0,\infty \rightopenint$,
\item if $r_0=0$, then $k_\iota=0$ and $\sigma_\iota$ is the Dirac delta at $0$,
\item if $r_0>0$, then $d\sigma_\iota(\mu) = \mu^{r_0/2+u_\iota - 1} \,d\mu$ for some $u_\iota \in \N$.
\end{itemize}
\end{prp}
\begin{proof}
Because of the support condition on $H$, both sides of the above inequality vanish if $\eta \notin \Omega$, hence we may assume $\eta \in \Omega$.

Under our assumption \eqref{eq:selfcontrolhypothesis}, we can estimate the ``differential polynomials'' in the right-hand side of \eqref{eq:derivatives} whenever $|\alpha| \leq A$. In fact, by Lemma~\ref{lem:sc}(iv), $\|1/b^\eta_\jone\|_{\SC^{|\alpha|}_{\Omega,D}} \leq C_{\kappa,\alpha}$, and consequently, by Lemma \ref{lem:hdp}(ii),
\[
|\Phi_{\iota,\jone}(D;1/b_\jone^\eta)| \leq C_{\kappa,\alpha} (1/b_\jone^\eta)^{\beta^\iota_\jone}.
\]
Analogously, since $\|P^\eta_\jone\|_{\BD^{|\alpha|}_{\Omega,D}(1)} \leq \kappa$ and
\[
|D^\theta |P^\eta_\jone \xi|^2| = |\langle (D^\theta P^\eta_\jone) \xi, \xi \rangle| \leq \|D^\theta P^\eta_\jone\| \, |\xi|^2,\\
\]
for all $\theta \in \N^N$, we deduce that
\[
|\Psi_{\iota,0}(D;|P_0^\eta \xi|^2)| \leq C_{\kappa,\alpha} |\xi|^{2k_\iota}, \qquad |\Psi^1_{\iota,\jone}(D;|P_\jone^\eta \xi|^2)| \leq C_{\kappa,\alpha} |\xi|^{2q^\iota_\jone}.
\]
For the terms $\Psi^0_{\iota,\jone}(D;|P_\jone^\eta\xi|^2)$, containing only derivatives of order zero or one, a better estimate holds, since
\[
|D_{\jtwo} |P^\eta_\jone \xi|^2| = |2\langle (D_{\jtwo} P^\eta_\jone) \xi, P^\eta_\jone \xi \rangle| \leq 2 \|D_{\jtwo} P^\eta_\jone\| \, |P^\eta_\jone \xi| \, |\xi|,
 \qquad |P^\eta_\jone \xi|^2 \leq |P^\eta_\jone \xi| \, |\xi|,
\]
and consequently
\[
|\Psi^0_{\iota,\jone}(D;|P_\jone^\eta \xi|^2)| \leq C_{\kappa,\alpha} |P_\jone^\eta \xi|^{\beta^\iota_\jone - q^\iota_\jone} |\xi|^{\beta^\iota_\jone - q^\iota_\jone}.
\]
From Proposition~\ref{prp:derivatives} and the triangular inequality we then obtain that
\begin{multline*}
\Bigl|D^\alpha V(\xi,\eta) \Bigr|^2 
\leq C_{\kappa,\alpha} \sum_{\iota \in I_{\alpha}} \Bigl| \sum_{n \in \N^\done}  D^{\gamma^\iota} \partial_\mu^{k_\iota} \delta^{\beta^\iota} m_H(n,|P_0^\eta \xi|^2,\eta) \\
\times |\xi|^{2k_\iota} \, \prod_{\jone=1}^\done \left[ \Ell^{(r_\jone-1+\beta^\iota_\jone)}_{n_\jone}(|P_\jone^\eta \xi|^2/b_\jone^\eta) \, (1/b_\jone^\eta)^{\beta^\iota_\jone} \, |\xi|^{\beta^\iota_\jone+q^\iota_\jone} |P^\eta_\jone \xi|^{\beta^\iota_\jone - q^\iota_\jone} \right] \Bigr|^2.
\end{multline*}
Since $|\xi|^2 = \sum_{\jone=0}^\done |P^\eta_\jone \xi|^2$, the sum can be rearranged so to give
\begin{multline*}
\Bigl|D^\alpha V(\xi,\eta) \Bigr|^2 
\leq C_{\kappa,\alpha} \sum_{\iota \in I'_{\alpha}} \Bigl| \sum_{n \in \N^\done} D^{\gamma^\iota} \partial_\mu^{k_\iota} \delta^{\beta^\iota} m_H(n,|P_0^\eta \xi|^2,\eta) \\
\times |P_0^\eta \xi|^{2\tilde k_\iota} \, \prod_{\jone=1}^\done \left[ \Ell^{(r_\jone-1+\beta^\iota_\jone)}_{n_\jone}(|P_\jone^\eta \xi|^2/b_\jone^\eta) \, (1/b_\jone^\eta)^{\beta^\iota_\jone} \, |P^\eta_\jone \xi|^{\beta^\iota_\jone - q^\iota_\jone + c^\iota_\jone} \right] \Bigr|^2,
\end{multline*}
where $\tilde k_\iota \in \N$, $c^\iota \in \N^\done$ and $|c^\iota| + 2\tilde k_\iota = |\beta^\iota| + |q^\iota| + 2k_\iota$. Set $p^\iota = \beta^\iota - q^\iota + c^\iota$, and let $\sigma_\iota$ be the measure on $\leftclosedint 0,\infty \rightopenint$ given by $\mu^{r_0/2-1+2\tilde k_\iota} \,d\mu$ if $r_0 > 0$, or by the Dirac measure in $0$ if $r_0 = 0$; then, by a change of variables,
\begin{multline*}
\int_{\fst} \Bigl|D^\alpha V(\xi,\eta) \Bigr|^2 \,d\xi 
\leq C_{\kappa,\alpha} \sum_{\iota \in I'_{\alpha}} \int_{\leftclosedint 0,\infty \rightopenint} \int_{\leftopenint 0,\infty \rightopenint^\done} \Bigl| \sum_{n \in \N^\done} D^{\gamma^\iota} \partial_\mu^{k_\iota} \delta^{\beta^\iota} m_H(n,\mu,\eta) \\
\times \prod_{\jone=1}^\done \Ell^{(r_\jone-1+\beta^\iota_\jone)}_{n_\jone}(t_\jone) \Bigr|^2 
\prod_{\jone=1}^\done \frac{t_\jone^{p^\iota_\jone+r_\jone-1}}{(b_\jone^\eta)^{2\beta^\iota_\jone-p^\iota_\jone-r_\jone}}  \,dt \,d\sigma_\iota(\mu),
\end{multline*}
which yields, by Lemma~\ref{lem:laguerreorthogonality},
\begin{multline*}
\int_{\fst} \Bigl|D^\alpha V(\xi,\eta) \Bigr|^2 \,d\xi
\leq C_{\kappa,\alpha} \sum_{\iota \in I'_{\alpha}} \int_{\leftclosedint 0,\infty \rightopenint} \sum_{n \in \N^\done} |D^{\gamma^\iota} \partial_\mu^{k_\iota} \delta^{\tilde\beta^\iota} m_H(n,\mu,\eta)|^2 \\
\times \prod_{\jone=1}^\done (b_\jone^\eta)^{1+a^\iota_\jone-2\tilde\beta^\iota_\jone} \langle n_\jone\rangle^{a^\iota_\jone} \,d\sigma_\iota(\mu),
\end{multline*}
where $\tilde\beta^\iota = \beta^\iota + (\beta^\iota-p^\iota)_+ \in \N^\done$ and $a^\iota_\jone = r_\jone-1+p^\iota_\jone+2(\beta^\iota_\jone-p^\iota_\jone)_+ \in \N$. On the other hand, since $(\beta^\iota-p^\iota)_+ = (q^\iota-c^\iota)_+ \leq q^\iota$, we have
\[
\min\{1,|\alpha|\} \leq |\gamma^\iota| + k_\iota + |\beta^\iota| \leq |\gamma^\iota| + k_\iota + |\tilde\beta^\iota| \leq |\gamma^\iota| + k_\iota + |\beta^\iota| + |q^\iota| \leq |\alpha|,
\]
and the conclusion follows by renaming $\tilde\beta^\iota$ as $\beta^\iota$ and then applying Lemma~\ref{lem:discretecontinuous}.
\end{proof}

\begin{cor}\label{cor:partiall2norm_factor}
Under the hypotheses of Proposition~\ref{prp:partiall2norm}, suppose further that
\[
H(\lambda,\eta) = F(\lambda) \, \chi(\eta);
\]
then
\begin{multline*}
\int_{\fst} \Bigl|D^\alpha V(\xi,\eta) \Bigr|^2 \, d\xi 
\leq C_{\kappa,\alpha} \sum_{\iota \in I''_{\alpha}} \int_{\leftclosedint 0,\infty \rightopenint} \int_{J_\iota} \sum_{n \in \N^\done} |D^{\gamma^\iota} \chi(\eta)|^2 \\
\times \left|F^{(k_\iota)}\left(\sum_{\jone=1}^\done (2(n_\jone+s_\jone)+r_\jone) b_\jone^\eta  + \mu\right)\right|^2 
\prod_{\jone=1}^\done \left[ (b^\eta_\jone)^{1+a^\iota_\jone} \langle n_\jone \rangle^{a^\iota_\jone} \right] \,d\nu_{\iota}(s) \,d\sigma_\iota(\mu),
\end{multline*}
for all $\eta \in \ddsnd$, where $I''_{\alpha}$ is a finite set and, for all $\iota \in I''_{\alpha}$,
\begin{itemize}
\item $a^\iota \in \N^\done$, $k_\iota \in \N$, $\gamma^\iota \in \N^N$,
\item $\gamma^\iota \leq \alpha$ and 
$\min\{1,|\alpha|\} \leq |\gamma^\iota| + k_\iota  \leq |\alpha|$,
\item $J_\iota = \prod_{\jone=1}^\done \leftclosedint 0,\beta^\iota_\jone \rightclosedint$ and $\nu_\iota$ is a Borel probability measure on $J_\iota$,
\item $\sigma_\iota$ is a regular Borel measure on $\leftclosedint 0,\infty \rightopenint$.
\end{itemize}
\end{cor}
\begin{proof}
This is an immediate consequence of Proposition~\ref{prp:partiall2norm} and the fact that, if $H(\lambda,\eta) = F(\lambda) \chi(\eta)$, then
\begin{multline*}
D^\gamma \partial_\mu^k \partial_n^\beta m_H(n,\mu,\eta) 
 = \sum_{\substack{\gamma' + \gamma'' \leq \gamma \\ |\gamma'| + |\gamma''| + |\beta| \geq \min\{1,|\gamma|\}}} F^{(|\beta|+k+|\gamma'|)}\left(\sum_{\jone=1}^\done b_\jone^\eta \langle n_\jone\rangle_\jone + \mu\right) \\
 \times D^{\gamma''}\chi(\eta) \, \prod_{\jone=1}^\done \Bigl[ \langle n_\jone\rangle_\jone^{\gamma'_\jone} \Psi_{\beta,\gamma,\gamma',\gamma'',\jone}(D;b_\jone^\eta) \Bigr]
\end{multline*}
where $\langle \ell \rangle_\jone = 2\ell+r_\jone$ and $\Psi_{\beta,\gamma,\gamma',\gamma'',\jone} \in \HDP_{N}^{|\gamma|}(\beta_\jone+\gamma'_\jone)$, so
\begin{multline*}
|D^\gamma \partial_\mu^k \partial_n^\beta m_H(n,\mu,\eta)|^2 
 \leq C_{\kappa,\gamma,k,\beta} \sum_{\substack{\gamma' + \gamma'' \leq \gamma \\ |\gamma'| + |\gamma''| + |\beta| \geq \min\{1,|\gamma|\}}} |D^{\gamma''}\chi(\eta)|^2 \\
 \times \prod_{\jone=1}^\done \Bigl[ (b_\jone^\eta)^{2\beta_\jone+2\gamma'_\jone} \langle n_\jone\rangle^{2\gamma'_\jone} \Bigr] \, \left| F^{(|\beta|+k+|\gamma'|)}\left(\sum_{\jone=1}^\done b_\jone^\eta \langle n_\jone\rangle_\jone + \mu\right) \right|^2
\end{multline*}
for all $n \in \R^\done$, $\mu \in \leftclosedint 0,\infty \rightopenint$, $\eta \in \ddsnd$, whenever $|\gamma| \leq A$.
\end{proof}

We are finally able to prove the fundamental estimate, which will allow us to carry on the strategy described in the introduction, based on the decomposition \eqref{eq:decomposition} of an operator $F(L)$ along the spectrum of $\vecU$. The following result shows in fact that a weighted $L^2$-norm of the kernel of $F(L) \, \chi(\vecU)$ for some cutoff $\chi$ can be controlled by a Sobolev norm of $F$ times the square root of a weighted measure of $\supp \chi$.

\begin{prp}\label{prp:cutoff_partiall2norm}
Suppose that $D = (\tilde\eta_1 \partial_{\tilde\eta_1},\dots,\tilde\eta_\dtwo \partial_{\tilde\eta_\dtwo})$ for some linear coordinates $(\tilde\eta_1,\dots,\tilde\eta_\dtwo)$ on $\dusnd$.
Let $\Omega \subseteq \ddsnd$ be open, and suppose that \eqref{eq:selfcontrolhypothesis} holds
for some $A \in \N$ and $\kappa \in \leftclosedint 0,\infty \rightopenint$.
Let $\chi \in C^\infty_c(\Omega)$ be of the form
\[
\chi(\eta) = \chi_\rad(f(\eta)) \, \chi_\sph(\eta),
\]
where $\chi_\rad \in C^\infty_c(\leftopenint 0,\infty \rightopenint)$, $f : \dusnd \setminus \{0\} \to \leftopenint 0,\infty \rightopenint$ is smooth and homogeneous of degree $1$, $\chi_\sph \in C^\infty(\dusnd \setminus \{0\})$ is homogeneous of degree $0$, and
\begin{gather}
\label{eq:suppcondition} |\supp\chi_\rad| \geq \kappa^{-1}, \qquad \supp \chi_\rad \subseteq \leftclosedint \kappa^{-1},\kappa \rightclosedint,\\
\label{eq:boundcondition}
\|\chi_\rad\|_{C^A},\|f\|_{\SC_{\Omega,D}^A}, \|\chi_\sph\|_{\BD_{\Omega,D}^A(1)} \leq \kappa.
\end{gather}
Let $(\tilde u_1,\dots,\tilde u_\dtwo)$ be the coordinates on $\snd$ dual to $(\tilde\eta_1,\dots,\tilde\eta_\dtwo)$.
Then, for all compact sets $K \subseteq \R$, for all Borel functions $F : \R \to \C$ supported in $K$, and for all $\alpha \in \N^\dtwo$ with $|\alpha| \leq A$,
\[
\int_G \left||\tilde u^\alpha| \, \Kern_{F(L) \, \chi(\vecU)}(x,u)\right|^2 \,dx\,du 
\leq C_{\kappa,K,\alpha} \|F\|^2_{W_2^{|\alpha|}} \int_{\supp\chi} |\tilde\eta^{-2\alpha}| \,d\eta.
\]
\end{prp}

Note that, in the above formula, the volume elements $du$ and $d\eta$ are fixed as in Corollary~\ref{cor:plancherelmeasure}, and do not depend on the choice of coordinates $(\tilde\eta_1,\dots,\tilde\eta_\dtwo)$ and $(\tilde u_1,\dots,\tilde u_\dtwo)$. On the other hand, by \eqref{eq:selfcontrolhypothesis} and \eqref{eq:boundcondition}, the quantity $\kappa$ will depend in general on such choice.

\begin{proof}
Via a standard approximation argument, we may assume that $F$ is smooth.

The Lebesgue measure on $\dusnd$ can be decomposed in ``polar coordinates'' according to $f$, i.e.,
\begin{equation}\label{eq:polarcoords}
\int_{\dusnd} \phi(\eta) \,d\eta = \int_0^\infty \int_S \phi(\rho\eta) \,d\Sigma(\eta) \,\rho^\dtwo \frac{d\rho}{\rho},
\end{equation}
for all (nonnegative or integrable) Borel $\phi : \dusnd \to \C$, where $S = \{\eta \tc f(\eta) = 1\}$ and $\Sigma$ is some regular Borel measure on $S$. Let $\tilde\chi_\rad : \R \to \R$ be the characteristic function of $\supp \chi_\rad$, and $\tilde\chi_\sph : S \to \R$ be the characteristic function of $\supp \chi_\sph \cap S$. Leibniz' rule, \eqref{eq:boundcondition} and Lemma~\ref{lem:sccomp2}(i) then yield
\[
|D^\gamma \chi(\eta)| \leq C_{\gamma,\kappa} \, \tilde\chi_\rad(f(\eta)) \, \tilde\chi_\sph(\eta/f(\eta))
\]
for all $\eta \in \dusnd \setminus \{0\}$ and $\gamma \in \N^N$ with $|\gamma| \leq A$. Note moreover that, if $\langle \ell \rangle_\jone = 2\ell + r_\jone$, then
\[
b_\jone^\eta \langle n_\jone \rangle \leq b_\jone^\eta \langle n_\jone + s_\jone \rangle_\jone \leq \sum_{\jone=1}^\done b_\jone^\eta \langle n_\jone + s_\jone \rangle_\jone + \mu \leq \max K
\]
for all $\eta \in \ddsnd$, $n \in \N^\done$, $s \in \leftclosedint 0,\infty \rightopenint^\done$, $\mu \in \leftclosedint 0,\infty \rightopenint$ such that $\sum_{\jone=1}^\done b_\jone^\eta \langle n_\jone + s_\jone \rangle_\jone + \mu \in \supp F$.
From Corollary~\ref{cor:partiall2norm_factor} we then deduce
\begin{multline*}
\int_{\fst} \Bigl|D^\alpha V(\xi,\eta) \Bigr|^2 \, d\xi 
\leq C_{\kappa,K,\alpha} \sum_{\iota \in I''_{\alpha}} \sum_{n \in \N^\done} \tilde\chi_\rad(f(\eta)) \, \tilde\chi_\sph(\eta/f(\eta)) \\
\times  \int_{\leftclosedint 0,\infty \rightopenint} \int_{J_\iota} \left|F^{(k_\iota)}\left(\sum_{\jone=1}^\done \langle n_\jone+s_\jone\rangle_\jone b_\jone^\eta  + \mu\right)\right|^2 
\prod_{\jone=1}^\done b^\eta_\jone \, \,d\nu_{\iota}(s) \,d\sigma_\iota(\mu),
\end{multline*}
for all $\eta \in \ddsnd$, where $I''_{\alpha}$ is a finite set and, for all $\iota \in I''_{\alpha}$,
\begin{itemize}
\item $k_\iota \in \N$, $k_\iota  \leq |\alpha|$,
\item $J_\iota = \prod_{\jone=1}^\done \leftclosedint 0,\beta^\iota_\jone \rightclosedint$ and $\nu_\iota$ is a Borel probability measure on $J_\iota$,
\item $\sigma_\iota$ is a regular Borel measure on $\leftclosedint 0,\infty \rightopenint$.
\end{itemize}
On the other hand it is easily proved that, for all $\alpha \in \N^N$,
\[
\partial_{\tilde\eta}^{\alpha} = \tilde\eta^{-\alpha} \sum_{\tilde\alpha \leq \alpha} c_{\alpha,\tilde\alpha} D^{\tilde\alpha}
\]
for some $c_{\alpha,\tilde\alpha} \in \Z$, hence
\begin{multline*}
\int_{\fst} \Bigl|\partial_{\tilde\eta}^\alpha V(\xi,\eta) \Bigr|^2 \, d\xi 
\leq C_{\kappa,K,\alpha} \,  |\tilde\eta^{-2\alpha}| \, \tilde\chi_\rad(f(\eta)) \, \tilde\chi_\sph(\eta/f(\eta)) \\
\times \sum_{\iota \in I'''_{\alpha}} \sum_{n \in \N^\done}  \int_{\leftclosedint 0,\infty \rightopenint} \int_{J_\iota} \left|F^{(k_\iota)}\left(\sum_{\jone=1}^\done \langle n_\jone+s_\jone\rangle_\jone b_\jone^\eta  + \mu\right)\right|^2 
\prod_{\jone=1}^\done b^\eta_\jone \, \,d\nu_{\iota}(s) \,d\sigma_\iota(\mu),
\end{multline*}
where $I'''_{\alpha}$ is the disjoint union of the $I''_{\tilde\alpha}$ with $\tilde\alpha\leq \alpha$. The properties of the Fourier transform give us finally
\begin{multline*}
\int_G \left||\tilde u^\alpha| \, \Kern_{F(L) \, \chi(\vecU)}(x,u)\right|^2 \,dx\,du 
\leq C_{\kappa,K,\alpha} \sum_{\iota \in I'''_{\alpha}} \sum_{n \in \N^\done} \int_{\ddsnd} \tilde\chi_\rad(f(\eta)) \, \tilde\chi_\sph(\eta/f(\eta)) \\
\times |\tilde\eta^{-2\alpha}| \int_{\leftclosedint 0,\infty \rightopenint} \int_{J_\iota} \left|F^{(k_\iota)}\left(\sum_{\jone=1}^\done \langle n_\jone+s_\jone\rangle_\jone b_\jone^\eta  + \mu\right)\right|^2 
\prod_{\jone=1}^\done b^\eta_\jone \, \,d\nu_{\iota}(s) \,d\sigma_\iota(\mu) \,d\eta.
\end{multline*}
Passing to polar coordinates as in \eqref{eq:polarcoords} and rescaling, we obtain that
\begin{multline*}
\int_G \left||\tilde u^\alpha| \, \Kern_{F(L) \, \chi(\vecU)}(x,u)\right|^2 \,dx\,du 
\leq C_{\kappa,K,\alpha} 
\\
\times \sum_{\iota \in I'''_{\alpha}} \int_0^\infty \int_{\supp \chi_\sph \cap S} |\tilde\eta^{-2\alpha}| \int_{J_\iota}  \sum_{n \in \N^\done} \tilde\chi_\rad\left(\frac{ \rho}{\sum_{\jone=1}^\done \langle n_\jone+s_\jone\rangle_\jone b_\jone^\eta} \right) \\
\times  \int_{\leftclosedint 0,\infty \rightopenint} |F^{(k_\iota)}\left(\rho  + \mu\right) |^2 
\prod_{\jone=1}^\done b^\eta_\jone \, \,d\sigma_\iota(\mu) \,d\nu_{\iota}(s) \,d\Sigma(\eta) \,\frac{d\rho}{\rho}.
\end{multline*}
In the sum over $\N^\done$ above, the $n$-th summand vanishes unless $\sum_{\jone=1}^\done \langle n_\jone+s_\jone\rangle_\jone b_\jone^\eta \leq \kappa\rho$, hence there are at most $(\kappa \rho)^\done \prod_{\jone=1}^\done (b_\jone^\eta)^{-1}$ nonvanishing summands, and our estimate becomes
\begin{multline*}
\int_G \left||\tilde u^\alpha| \, \Kern_{F(L) \, \chi(\vecU)}(x,u)\right|^2 \,dx\,du 
\leq C_{\kappa,K,\alpha} \int_{\supp\chi_\sph \cap S} |\tilde\eta^{-2\alpha}|  \,d\Sigma(\eta) \\
\times \sum_{\iota \in I'''_{\alpha}}  
 \int_{\leftclosedint 0,\infty \rightopenint} \int_0^\infty |F^{(k_\iota)}\left(\rho  + \mu\right) |^2 
\rho^{\done-1}  \,d\rho \,d\sigma_\iota(\mu).
\end{multline*}
On the other hand
\begin{multline*}
\int_{\leftclosedint 0,\infty \rightopenint} \int_0^\infty |F^{(k_\iota)}\left(\rho  + \mu\right) |^2 
\rho^{\done-1}  \,d\rho \,d\sigma_\iota(\mu) \\
\leq C_{K} \, \sigma_\iota(\leftclosedint 0,\max K \rightclosedint) \sup_{\mu \in \leftclosedint 0,\max K\rightclosedint} \int_0^\infty |F^{(k_\iota)}\left(\rho+\mu\right) |^2 \,d\rho \leq C_{K,\alpha} \|F\|_{W_2^{|\alpha|}}
\end{multline*}
for all $\iota \in I'''_\alpha$, and moreover
\begin{multline*}
\int_{\supp \chi} |\tilde\eta^{-2\alpha}| \,d\eta = \int_0^\infty \int_S \tilde\chi_\rad(\rho) \, \tilde\chi_\sph(\eta) \, |\tilde\eta^{-2\alpha}| \, \rho^{\dtwo-2|\alpha|} \,d\Sigma(\eta) \,\frac{d\rho}{\rho}  \\
\geq C_\kappa \, \int_{\supp\chi_\sph \cap S} |\tilde\eta^{-2\alpha}| \,d\Sigma(\eta)
\end{multline*}
by \eqref{eq:polarcoords} and \eqref{eq:suppcondition}, and the conclusion follows.
\end{proof}

\section{Groups with $2$-dimensional second layer}\label{section:2dc}

By using the estimates obtained in \S\ref{section:kernelderivatives}, here we prove Proposition~\ref{prp:l1low} in the case $\dtwo = \dim\snd \leq 2$. In fact, if $\dtwo = 1$, then $G$ is a Heisenberg group, and Proposition~\ref{prp:l1low} follows from the results of \cite{mller_spectral_1994}. Therefore in the rest of the section we suppose that $\dtwo = 2$.

We now show that the singular set $\dusnd \setminus \ddsnd$ is the union of a finite number of rays emanating from the origin, and that in the neighborhood of each of these rays a system of coordinates on $\dusnd$ can be chosen so to satisfy the hypothesis \eqref{eq:selfcontrolhypothesis}.

Fix any Euclidean norm and orientation on $\dusnd$, and denote by $S$ the unit sphere in $\dusnd$. For all $v \in S$, let $(\eta_1^v,\eta_2^v)$ denote the coordinates on $\dusnd$ determined by completing $v$ to a positive orthonormal basis of $\dusnd$, and set $D_v = (\eta_1^v \partial_{\eta_1^v},\eta_2^v \partial_{\eta_2^v})$. For all $X \subseteq \dusnd \setminus \{0\}$, let $\cone{X}$ denote the ``cone over $X$'', i.e., the set $\{ \lambda v \tc \lambda \in \leftopenint 0, \infty \rightopenint, \, v \in X\}$.

\begin{lem}\label{lem:2dc_discretesingularities}
There exists a finite subset $N$ of $S$ such that
\[
\dusnd \setminus \ddsnd = \{0\} \cup \cone{N}.
\]
Moreover, if $U$ is an open subset of $S \setminus N$ and $\overline{U} \cap N \subseteq \{v\}$ for some $v \in S$, then 
\[
b_1^\eta,\dots,b_\done^\eta \in \SC_{\cone{U},D_v}^\infty \quad\text{and}\quad P_0^\eta,P_1^\eta,\dots,P_\done^\eta \in \BD_{\cone{U},D_v}^\infty(1).
\]
\end{lem}
\begin{proof}
By Lemma~\ref{lem:factorization}, $\dusnd \setminus \ddsnd$ is the zero set of a nonzero homogeneous polynomial $T$ on $\dusnd$, hence it corresponds to the zero set of $T$ in the projective space $P(\dusnd)$. Since $\dusnd$ is $2$-dimensional, $P(\dusnd)$ is $1$-dimensional, hence the zero set of $T$ in $P(\dusnd)$ is finite, and $\dusnd \setminus \ddsnd = \{0\} \cup \cone{N}$ for some finite subset $N$ of $S$.

Let now $U$ be an open subset of $S$ such that $\overline{U} \cap N \subseteq \{v\}$ for some $v \in S$. If $v \notin N$, then $\overline{U} \cap N = \emptyset$. In particular the functions $b_\jone^\eta$ and $P_\jone^\eta$ are smooth in a neighborhood of $\cone{\overline{U}}$, and the $b_\jone^\eta$ do not vanish there; their homogeneity properties are then sufficient to conclude, by Lemma~\ref{lem:schom_unif}, that $b_j^\eta \in \SC^\infty_{\cone{U},D_v}$ and $P_j^\eta \in \BD^\infty_{\cone{U},D_v}(1)$.

Suppose instead that $v \in N$. Let us use the coordinates $(\eta_1^v,\eta_2^v)$ on $\dusnd$: then $v$ corresponds to the point $(1,0)$, and $T$ becomes a nonzero homogeneous polynomial in two indeterminates with $T(1,0) = 0$. Denote by $p_\eta$ the characteristic polynomial of $-J_\eta^2$, as in \eqref{eq:charpoly}, and let $\tilde p_t(\lambda) = p_{(1,t)}(\lambda)$, $\tilde T(t) = T(1,t)$, $\tilde b_j(t) = b_j^{(1,t)}$, $\tilde P_j(t) = P_j^{(1,t)}$. Then the $\tilde b_j(t)$ are the square roots of the roots of $\tilde p_t$, i.e.,
\[\tilde p_t(\lambda) = \lambda^{r_0} (\lambda - (\tilde b_1(t))^2)^{2r_1} \cdots (\lambda - (\tilde b_\done(t))^2)^{2r_\done},\]
and are analytic functions on $\{t \tc \tilde T(t) \neq 0\}$. Moreover, since the coefficients of $\tilde p_t$ are polynomials in $t$, by Puiseux's theorem on the resolution of singularities of plane algebraic curves (see, e.g., \cite[\S 7]{fischer_plane_2001} or the discussion in \cite[\S 3]{phong_newton_1997}) there exists $\epsilon > 0$ such that the functions $\tilde b_\jone$, restricted to $\leftopenint 0,\epsilon \rightopenint$, admit a convergent Puiseux expansion; by Lemma~\ref{lem:spectraldecomposition}, the same is true for the matrix coefficients of the $\tilde P_\jone$, because they are rational functions of $\tilde b_1,\dots,\tilde b_\done$ and the identity.
This means that, if $f$ denotes any of the $\tilde b_j$ or any of the matrix coefficients of the $\tilde P_j$, then $f$ on the interval $\leftopenint 0,\epsilon \rightopenint$ can be written as
\[
f(t) = t^{h/n} \sum_{m \geq 0} a_m t^{m/n}
\]
for some $n \in \N \setminus \{0\}$, $h \in \Z$, and coefficients $a_m \in \R$ with $a_0 \neq 0$ (in fact it must be $h \in \N$, because both the roots $\tilde b_j(t)$ and the coefficients of the projections $\tilde P_j(t)$ are bounded in a neighborhood of $t = 0$); term by term differentiation then gives that
\[
(t\partial_t)^k f(t) = t^{h/n} \sum_{m \geq 0} (h+m)^k n^{-k} a_m t^{m/n}
\]
for all $k \in \N$ (note that the derived series have the same radius of convergence). Therefore, for all $k \in \N$, the function $t \mapsto t^{-h/n} (t\partial_t)^k f(t)$ has a continuous extension to $\leftclosedint 0,\epsilon \rightopenint$; moreover, since $a_0 \neq 0$, modulo taking a smaller $\epsilon$, we may assume that the continuous extension of $t \mapsto t^{-h/n} f(t)$ does not vanish in $\leftclosedint 0,\epsilon \rightopenint$. Consequently, by compactness,
\[
|t^{-h/n} (t\partial_t)^k f(t)| \leq C_{f,k} |t^{-h/n} f(t)|
\]
for all $t \in \leftclosedint 0,\epsilon/2 \rightclosedint$, and therefore
\begin{equation}\label{eq:homogeneousinequality}
|(t\partial_t)^k f(t)| \leq C_{f,k} |f(t)|
\end{equation}
for all $t \in \leftopenint 0,\epsilon/2 \rightclosedint$. The same argument, applied to the function $t \mapsto f(-t)$, shows that \eqref{eq:homogeneousinequality} holds for all $t \in \leftclosedint -\epsilon/2,\epsilon/2 \rightclosedint \setminus \{0\}$ and all $k \in \N$, if $\epsilon > 0$ is sufficiently small.

Note now that, if $F$ is one of the $\eta \mapsto b_j^\eta$ or one of the matrix coefficients of the $\eta \mapsto P_j^\eta$, then $F(\eta_1^v,\eta_2^v) = (\eta_1^v)^\mu f(\eta_2^v/\eta_1^v)$ for some $\mu \in \{0,1\}$, where $f(t) = F(1,t)$. Inductively one then shows that
\[
D_v^\alpha F(\eta_1^v,\eta_2^v) = \sum_{s=0}^{|\alpha|} c_{\alpha,s} \, (\eta_1^v)^\mu ((t\partial_t)^s f)(\eta_2^v/\eta_1^v)
\]
for all $\alpha \in \N^2$ and some coefficients $c_{\alpha,s} \in \R$, and in particular, by \eqref{eq:homogeneousinequality},
\begin{equation}\label{eq:homogeneousinequality2}
|D_v^\alpha F(\eta_1^v,\eta_2^v)| \leq C_{\alpha,F} |F(\eta_1^v,\eta_2^v)|
\end{equation}
for all $\alpha \in \N^2$ and all $(\eta_1^v,\eta_2^v)$ with $0 < |\eta_2^v/\eta_1^v| \leq \epsilon/2$.

On the other hand, if $V = \{ \eta \in U \tc |\eta_2^v/\eta_1^v| > \epsilon/2\}$, then $\overline{V} \cap N = \emptyset$ and we already know that $b_j^\eta \in \SC^\infty_{\cone{V},D_v}$ and $P_j^\eta \in \BD^\infty_{\cone{V},D_v}(1)$. By combining this fact with \eqref{eq:homogeneousinequality2} and the boundedness of the coefficients of $P_j^\eta$, we obtain that $b_j^\eta \in \SC^\infty_{\cone{U},D_v}$ and $P_j^\eta \in \BD^\infty_{\cone{U},D_v}(1)$.
\end{proof}

Let $N$ be the finite subset of $S$ given by Lemma~\ref{lem:2dc_discretesingularities}; in the case this set is empty, we take instead $N = \{(1,0)\}$. We may then choose an open cover $\{U_v\}_{v \in N}$ of $S$ such that $\overline{U_v} \cap N = \{v\}$ for all $v \in N$.

Let $\{\zeta_v\}_{v \in N}$ be a smooth partition of unity on $S$ subordinate to the open cover $\{U_v\}_{v \in N}$; each $\zeta_v$ extends uniquely to a smooth function on $\dusnd \setminus \{0\}$, homogeneous of degree $0$, which we still denote by $\zeta_v$. Let moreover $\chi \in C^\infty_c(\leftopenint 0,\infty \rightopenint)$ be such that $\supp \chi \subseteq \leftclosedint 1/2, 2 \rightclosedint$ and $\sum_{n \in \Z} \chi(2^n t) = 1$ for all $t \in \leftopenint 0,\infty \rightopenint$. For all $v \in N$, denote by $(u_1^v,u_2^v)$ the coordinates on $\snd$ dual to the coordinates $(\eta_1^v,\eta_2^v)$ on $\dusnd$.

\begin{prp}\label{prp:2dc_secondweightedestimate}
Let $K \subseteq \R$ be compact. For all Borel functions $F : \R \to \C$ such that $\supp F \subseteq K$,
for all $v \in N$,
and for all $\alpha \in \leftclosedint 0,1/2 \rightopenint^2$,
\[
\int_{G} | (1+|u_1^v|)^{\alpha_1} (1+|u_2^v|)^{\alpha_2} \Kern_{F(L) \, \zeta_{v}(\vecU)}(x,u) |^2 \,dx \,du \leq C_{K,\alpha} \|F\|_{W_2^{\alpha_1+\alpha_2}}^2.
\]
\end{prp}
\begin{proof}
For all $\radp \in \leftopenint 0,\infty \rightopenint$ and $\crdp = (\crdp_1,\crdp_2) \in \leftopenint 0, 1 \rightclosedint^2$, let $\chi_{v,\radp,\crdp} : \dusnd \to \C$ be defined by
\[
\zeta_{v,\radp,\crdp}(\lambda,\eta) = \chi(|\eta|/\radp) \, \zeta_v(\eta/|\eta|) \, \prod_{\jtwo = 1}^2 \chi(|\eta^v_\jtwo|/(|\eta| \crdp_\jtwo)).
\]
Then 
\[
\zeta_v(\eta) = \sum_{m \in \Z, \, n \in \N^2} \zeta_{v,2^m,(2^{-n_1},2^{-n_2})}(\eta)
\]
for all $\eta \in \dusnd$ with $\eta_1^v \eta_2^v \neq 0$ (in fact some of the summands are identically zero, but we may disregard this). On the other hand the set $\{\eta \tc \eta_1^v \eta_2^v = 0\}$ is negligible with respect to the joint spectral resolution of $\vecU$. Moreover, by Lemma~\ref{lem:spectraldecomposition}, $\eta \mapsto (\sum_\jone 2r_\jone (b_\jone^\eta)^2)^{1/2}$ is a norm on $\dusnd$, hence there is a constant $\kappa \in \leftopenint 0,\infty \rightopenint$ such that
\[
|\eta| \leq \kappa \sum_\jone b_\jone^\eta \leq \kappa \left(\sum_\jone b_\jone^\eta (2n_\jone+r_\jone) + \mu\right)
\]
for all $\eta \in \dusnd$, $\mu \in \leftclosedint 0,\infty \rightopenint$ and $n \in \N^\done$; hence, by \eqref{eq:kernel}, $F(L) \, \zeta_{v,\radp,\crdp}(\vecU) = 0$ unless $\radp \leq 2\kappa \max K$. Consequently
\begin{equation}\label{eq:2dc_dyadic}
F(L) \, \zeta_v(\vecU) = \sum_{\substack{m \in \Z, \, n \in \N^2 \\ 2^{m} \leq 2\kappa \max K}} F(L) \, \zeta_{v,2^m,(2^{-n_1},2^{-n_2})}(\vecU)
\end{equation}
in the strong $L^2$ operator topology.

Set $\Omega_v = \cone{U_v} \cap \{ \eta \tc \eta_1^v \eta_2^v \neq 0\}$. If $f$ is any of the functions $\eta \mapsto |\eta|/\radp$, $\eta \mapsto |\eta_\jtwo|/(|\eta| \crdp_\jtwo)$, then by \eqref{eq:sc_scalar} and Lemma~\ref{lem:schom_multi} it is immediately seen that, for all $A \in \N$, $\|f\|_{\SC_{\Omega_v,D_v}^A}$ is finite and independent of $\radp,\crdp$. Therefore, if
\[
\zeta_{v,\crdp}(\eta) = \zeta_v(\eta/|\eta|) \, \prod_{\jtwo = 1}^2 \chi(|\eta^v_\jtwo|/(|\eta| \crdp_\jtwo)),
\]
then $\sup_{\crdp \in \leftopenint 0,1 \rightclosedint^2} \|\zeta_{v,\crdp}\|_{\BD^A_{\Omega_v,D_v}(1)}$ is finite by Lemma~\ref{lem:sccomp2}(i), and
\[
\zeta_{v,\radp,\crdp}(\lambda,\eta) = \chi(|\eta|/\radp) \, \zeta_{v,\crdp}(\eta).
\]
On the other hand, $|\eta^v_\jtwo| \sim \radp\crdp_\jtwo$ for $\eta \in \supp\zeta_{v,\radp,\crdp}$, hence the measure of $\supp\zeta_{v,\radp,\crdp}$ is at most $\radp^2 \crdp_1 \crdp_2$.
Consequently, by Proposition~\ref{prp:cutoff_partiall2norm} and Lemma~\ref{lem:2dc_discretesingularities}, for all $\alpha \in \N^2$,
\[
\int_G \left||u_1^v|^{\alpha_1} |u_2^v|^{\alpha_2} \, \Kern_{F(L) \, \zeta_{v,\radp,\crdp}(\vecU)}(x,u)\right|^2 \,dx\,du 
\leq C_{K,\kappa,\alpha} \|F\|^2_{W_2^{|\alpha|}} \prod_{\jtwo=1}^2 (\radp \crdp_\jtwo)^{1-2\alpha_\jtwo}.
\]
By interpolation, the same estimate holds for all $\alpha \in \leftclosedint 0,\infty \rightopenint^2$. In the case $\alpha_1,\alpha_2 < 1/2$, the dyadic decomposition \eqref{eq:2dc_dyadic} then yields
\[
\int_{G} | |u_1^v|^{\alpha_1} |u_2^v|^{\alpha_2} \Kern_{F(L) \, \zeta_v(\vecU)}(x,u) |^2 \,dx \,du \leq C_{K,\alpha} \|F\|_{W_2^{\alpha_1+\alpha_2}}^2.
\]
In order to conclude, it is sufficient to combine this estimate with the ones where $(\alpha_1,\alpha_2)$ is replaced by $(0,0)$, $(\alpha_1,0)$, $(0,\alpha_2)$.
\end{proof}

Via H\"older's inequality, the previous estimates can be combined at the level of $L^1$; interpolation with the standard estimate valid on all $2$-step groups then allows us to conclude the proof of Proposition~\ref{prp:l1low} for the groups with $\dtwo = 2$.

\begin{prp}\label{prp:2dc_l1estimate}
Let $K \subseteq \R$ be compact. For all Borel functions $F : \R \to \C$ such that $\supp F \subseteq K$, and for all $\alpha,\beta \in \R$ such that $\beta \geq (\dim \snd)/2$ and $\beta > \alpha + d/2$,
\begin{equation}\label{eq:weightedl1}
\| (1+|\cdot|_G)^\alpha \, \Kern_{F(L)} \|_1 \leq C_{K,\alpha,\beta} \| F \|_{W_2^\beta}.
\end{equation}
\end{prp}
\begin{proof}
We prove \eqref{eq:weightedl1} for $\alpha,\beta$ belonging to two different ranges:
\begin{gather}
\label{eq:range1} \beta \geq 0, \qquad \beta > \alpha + Q/2;\\
\label{eq:range2} \beta \geq 0, \quad 2\beta > \alpha + Q/2, \quad \alpha < -(\dim \fst)/2.
\end{gather}
The conclusion is then obtained by interpolation: in fact, for all small $\delta > 0$, the point $P_\delta$ with coordinates $\alpha=-\dim\fst/2-\delta,\beta=\dim\snd/2$ belongs to \eqref{eq:range2} and also to the line $\beta = \alpha+d/2+\delta$, hence the convex hull of $P_\delta$ and the region \eqref{eq:range1} contains the range $\beta \geq (\dim \snd)/2$, $\beta > \alpha+d/2+\delta$. 

For the range \eqref{eq:range1}, we choose $s \in \leftopenint \alpha+Q/2, \beta \rightopenint$ and then apply H\"older's inequality and the standard estimate \eqref{eq:standardl2}:
\[
\| (1+|\cdot|_G)^\alpha \, \Kern_{F(L)} \|_1 \leq \| (1+|\cdot|_G)^{\alpha-s} \|_2 \, \| (1+|\cdot|_G)^s \, \Kern_{F(L)} \|_2 \leq C_{K,\alpha,\beta} \|F\|_{W_2^\beta}.
\]

For the range \eqref{eq:range2}, instead, we first split the left-hand side of \eqref{eq:weightedl1} as follows:
\[
\| (1+|\cdot|_G)^\alpha \, \Kern_{F(L)} \|_1 \leq \sum_{v \in N} \| (1+|\cdot|_G)^\alpha \, \Kern_{F(L) \, \zeta_v(\vecU)} \|_1.
\]
Each summand in the right-hand side can be then estimated by H\"older's inequality: for all $\theta \in \R$,
\begin{multline*}
\| (1+|\cdot|_G)^\alpha \, \Kern_{F(L) \, \zeta_v(\vecU)} \|_1 \\
\leq \left(\int_G (1+|(x,u)|_G)^{2\alpha} (1+|u_1^v|)^{-2\theta} (1+|u_2^v|)^{-2\theta} \,dx \,du \right)^{1/2} \\
\times \left(\int_G |(1+|u_1^v|)^{\theta} (1+|u_2^v|)^{\theta} \Kern_{F(L) \, \zeta_v(\vecU)}(x,u)|^2 \,dx \,du \right)^{1/2}.
\end{multline*}
Note that $\alpha + (\dim \fst)/2 < 0$, and therefore $\alpha + Q/2 < \dim \snd = 2$. Choose $\theta$ so that $4\theta \in \leftclosedint0,2\beta\rightclosedint \cap \leftopenint \alpha+Q/2,2\rightopenint$, then choose $\alpha_1 \in \leftopenint (\dim \fst)/2,-\alpha+4\theta-2\rightopenint$, and set $\alpha_2 = -\alpha-\alpha_1$. Hence $-\alpha = \alpha_1 + \alpha_2$, $\alpha_1 > (\dim\fst)/2 > 0$ and $\alpha_2 > 2-4\theta > 0$. Therefore
\begin{multline*}
\int_G (1+|(x,u)|_G)^{2\alpha} (1+|u_1^v|)^{-2\theta} (1+|u_2^v|)^{-2\theta} \,dx \,du  \\
\leq \int_G (1+|x|)^{-2\alpha_1} (1+|u_1^v|)^{-2\theta-\alpha_2/2} (1+|u_2^v|)^{-2\theta-\alpha_2/2} \,dx \,du < \infty,
\end{multline*}
since $2\alpha_1 > \dim \fst$ and $2\theta + \alpha_2/2 > 1$. On the other hand, by Proposition~\ref{prp:2dc_secondweightedestimate},
\[
\int_G |(1+|u_1^v|)^{\theta} (1+|u_2^v|)^{\theta} \Kern_{F(L) \, \zeta_v(\vecU)}(x,u)|^2 \,dx \,du \leq C_{K,\alpha,\beta} \|F\|^2_{W_2^{2\theta}} \leq C_{K,\alpha,\beta} \|F\|^2_{W_2^{\beta}}
\]
since $\theta \in \leftclosedint 0,1/2\rightopenint$ and $2\theta \leq \beta$.
\end{proof}

\section{Groups of dimension at most $7$}\label{section:7dim}

In view of the results of \S\ref{section:2dc}, in order to complete the proof of Proposition~\ref{prp:l1low}, it remains to consider the case $d \leq 7$ and $\dim\snd > 2$. Some remarks on the possible structures of $G$ will help us to identify the cases which are not already covered by the existing literature.

\begin{prp}
Suppose that $d \leq 7$. Then
\begin{itemize}
\item[(i)] $\dim\snd \leq 3$;
\item[(ii)] if $d < 7$ and $\dim\snd = 3$, then $G$ is isomorphic to the free $2$-step nilpotent group $N_{3,2}$ on $2$ generators;
\item[(iii)] if $d = 7$, $\dim\snd = 3$, and $\lie{g}$ is decomposable, then $G$ is isomorphic to the direct product $N_{3,2} \times \R$, and the sublaplacian $L$ decomposes as $L'+L''$, where $L'$ and $L''$ correspond to sublaplacians on the factors $N_{3,2}$ and $\R$.
\end{itemize}
\end{prp}
\begin{proof}
Since $G$ is a quotient of the free $2$-step nilpotent group on $\dim\fst$ generators, it must be $\dim\snd \leq \binom{\dim\fst}{2}$, and the assumption $d \leq 7$ implies that $\dim\snd \leq 3$.

For the same reason, if $\dim\snd = 3$, then $d \geq 6$; in the case $d < 7$ we conclude that $d = 6$ and that $G$ is isomorphic to $N_{3,2}$.

Suppose now that $d = 7$, $\dim\snd = 3$, and $\lie{g}$ is decomposable, that is, $\lie{g} = \lie{g}' \oplus \lie{g}''$ for some nontrivial (commuting) ideals $\lie{g}',\lie{g}''$ of $\lie{g}$. In particular $\snd = [\lie{g}',\lie{g}'] \oplus [\lie{g}'',\lie{g}'']$, and modulo replacing $\lie{g}'$ with $\lie{g}''$ we may assume $\dim[\lie{g}',\lie{g}'] \geq 2$. But then $\dim\lie{g}' \geq 5$, hence $\dim\lie{g}'' \leq 2$, therefore $\lie{g}''$ is abelian, thus necessarily $[\lie{g}',\lie{g}'] = \snd$, which is $3$-dimensional, and consequently $\dim\lie{g}' \geq 6$; since $\lie{g}',\lie{g}''$ are nontrivial, it must be $\dim\lie{g}'=6$, $\dim\lie{g}'' = 1$. Therefore, if $\lie{z}$ is the center of $\lie{g}$, then $\dim\lie{z} = 4$, while $\dim\snd = 3$, and since $\lie{z} = (\lie{z} \cap \lie{g}_1) \oplus \lie{g}_2$, then $\dim(\lie{z} \cap \fst) = 1$. Let $\fst'$ be the orthogonal complement of $\lie{z} \cap \fst$ in $\fst$. Then $[\fst',\fst'] = \snd$. Consequently we obtain the decomposition $\lie{g} = \tilde{\lie{g}}'\oplus \tilde{\lie{g}}''$, where $\tilde{\lie{g}}' = \fst'\oplus\snd$ and $\tilde{\lie{g}}'' = \fst \cap \lie{z}$ are commuting ideals, the Lie algebra $\tilde{\lie{g}}'$ is isomorphic to the Lie algebra of $N_{3,2}$, and $\tilde{\lie{g}}''$ is $1$-dimensional. By choosing an orthonormal basis of $\fst$ adapted to the decomposition $\fst = \fst' \oplus (\fst \cap \lie{z})$, we can write $L = L' + L''$, where $L'$ and $L''$ correspond to sublaplacians on $\tilde{\lie{g}}'$ and $\tilde{\lie{g}}''$ respectively.
\end{proof}

By part (i) of the previous proposition, our assumptions $d \leq 7$ and $\dim\snd > 2$ imply $\dim\snd = 3$. Moreover, by part (ii), the case $\dim\snd=3$ and $d<7$ is covered by \cite{martini_n32}, whereas, by part (iii), the case $d=7$, $\dim\snd=3$ and $\lie{g}$ decomposable is covered by \cite[\S 6]{martini_heisenbergreiter}.

It remains to consider the case $\lie{g}$ indecomposable, $d = 7$, $\dim\fst = 4$, $\dim\snd = 3$. Fix an identification of $\fst$ with $\R^4$, so that the inner product on $\fst$ determined by the sublaplacian becomes the standard inner product on $\R^4$. The map $\eta \mapsto J_\eta$ then determines an embedding of $\dusnd$ in $\lie{so}_4$, the space of $4\times 4$ skewsymmetric real matrices. It is then convenient to analyze the spectral decomposition of the elements of $\lie{so}_4$.

The identification of $\R^4$ with $\C^2$ allows us to identify $\lie{su}_2$ with a subspace of $\lie{so}_4$. If $K$ is the $\R$-linear involutive automorphism of $\C^2$ given by $(z_1,z_2) \mapsto (z_1,\overline{z_2})$ and $\widetilde{\lie{su}_2} = K \lie{su}_2 K$, then
\begin{equation}\label{eq:semisimpledec}
\lie{so}_4 = \lie{su}_2 \oplus \widetilde{\lie{su}_2}
\end{equation}
is the decomposition of the semisimple Lie algebra $\lie{so}_4$ into simple ideals. Let $\mu = \mu^- + \mu^+$ denote the decomposition of an element $\mu \in \lie{so}_4$ according to \eqref{eq:semisimpledec}.
Fix moreover the inner product on $\lie{so}_4$ defined by
\[ 
\langle \mu,\mu' \rangle = -\tr(\mu \mu')/4
\]
for all $\mu,\mu' \in \lie{so}_4$, and let $|\cdot|$ denote the corresponding norm.

\begin{prp}\label{prp:four_spectral}
Let $\mu \in \lie{so}_4$.
\begin{itemize}
\item[(i)] If $\mu^+ = 0$ or $\mu^- = 0$, then $-\mu^2 = |\mu|^2$.
\item[(ii)] If both $\mu^+,\mu^-$ are nonzero, then
\[
P_1^\mu = \frac{1}{2} - \frac{1}{2} \frac{\mu^+}{|\mu^+|} \frac{\mu^-}{|\mu^-|}, \qquad P_2^\mu = \frac{1}{2} + \frac{1}{2} \frac{\mu^+}{|\mu^+|} \frac{\mu^-}{|\mu^-|}
\]
are complementary orthogonal projections on $\R^4$, and if
\[
b_1^\mu = |\mu^+|+|\mu^-|, \qquad b_2^\mu = ||\mu^+|-|\mu^-||,
\]
then
\[
-\mu^2 = (b_1^\mu)^2 P_1^\mu + (b_2^\mu)^2 P_2^\mu.
\]
\end{itemize}
\end{prp}
\begin{proof}
Note that, for all $\mu \in \lie{so}_4$ and $\Omega \in SO_4$, $(\Omega \mu \Omega^{-1})^\pm = \Omega \mu^\pm \Omega^{-1}$ and $|\Omega \mu \Omega^{-1}| = |\mu|$. Via these identities, we may reduce to the case where the skewsymmetric matrix $\mu$ is in normal form, i.e., $\mu = \lambda_1 I_1 + \lambda_2 I_2$, where $\lambda_1,\lambda_2 \in \R$, $\lambda_1 \geq |\lambda_2|$, and
\[
I_1 = \begin{pmatrix}
0 & -1 & 0 & 0 \\
1 &  0 & 0 & 0 \\
0 &  0 & 0 & 0 \\
0 &  0 & 0 & 0 
\end{pmatrix}, \qquad
I_2 = \begin{pmatrix}
0 & 0 & 0 &  0 \\
0 & 0 & 0 &  0 \\
0 & 0 & 0 & -1 \\
0 & 0 & 1 &  0 
\end{pmatrix}.
\]
Set $I_\pm = I_1 \pm I_2$. It is then easily seen that $I_- \in \lie{su}_2$, $I_+ \in \widetilde{\lie{su}_2}$, $I_\pm^2 = -1$ and $|I_\pm| = 1$; therefore $\mu^\pm = (\lambda_1\pm\lambda_2) I_\pm /2$ and $|\mu^\pm| = (\lambda_1\pm\lambda_2)/2$. From this, part (i) follows immediately. As for part (ii), a simple computation shows that
\[
\frac{\mu^+}{|\mu^+|} \frac{\mu^-}{|\mu^-|} = I_+ I_- = \begin{pmatrix}
-1 &  0 & 0 & 0 \\
 0 & -1 & 0 & 0 \\
 0 &  0 & 1 & 0 \\
 0 &  0 & 0 & 1 
\end{pmatrix},
\]
and therefore
\[
\frac{1}{2} - \frac{1}{2} \frac{\mu^+}{|\mu^+|} \frac{\mu^-}{|\mu^-|} = \begin{pmatrix}
1 &  0 & 0 & 0 \\
 0 & 1 & 0 & 0 \\
 0 &  0 & 0 & 0 \\
 0 &  0 & 0 & 0 
\end{pmatrix},
\quad
\frac{1}{2} + \frac{1}{2} \frac{\mu^+}{|\mu^+|} \frac{\mu^-}{|\mu^-|} = \begin{pmatrix}
0 &  0 & 0 & 0 \\
 0 & 0 & 0 & 0 \\
 0 &  0 & 1 & 0 \\
 0 &  0 & 0 & 1 
\end{pmatrix},
\]
which are complementary orthogonal projections; since moreover $|\mu^+| + |\mu^-| = \lambda_1$ and $|\mu^+|-|\mu^-| = \lambda_2$, the conclusion follows.
\end{proof}

The previous proposition further reduces the cases to be considered. In fact, by (i), if the image $V$ of $\dusnd$ via $\eta \mapsto J_\eta$ coincides with one of the $3$-dimensional subspaces $\lie{su}_2, \widetilde{\lie{su}_2}$ of $\lie{so}_4$, then $-J_\eta^2$ is a multiple of the identity for all $\eta \in \dusnd$, that is, $G$ is an H-type group, and this case is covered by \cite{hebisch_multiplier_1993}. On the other hand, if $V$ is contained in the ``cone'' $C = \{ \mu \tc |\mu^+| = |\mu^-| \}$, then by (ii) $-J_\eta^2$ has exactly one nonzero eigenvalue for all $\eta \in \dusnd \setminus \{0\}$, hence this case is covered by \cite{martini_heisenbergreiter}.

We can then suppose that $V_- = V \cap \lie{su}_2$ and $V_+ = V \cap \widetilde{\lie{su}_2}$ are proper subspaces of $V$, and that $V \cap C$ is a proper Zariski-closed subset of $V$. Hence, if we set
\[
\ddsnd = \{ \eta \in \dusnd \tc 0 \neq |J_\eta^-| \neq |J_\eta^+| \neq 0 \},
\]
then $\dusnd \setminus \ddsnd$ is a proper Zariski-closed subset of $\dusnd$ and, for all $\eta \in \ddsnd$, the spectral decomposition of $-J_\eta^2$ as in Lemma~\ref{lem:spectraldecomposition} can be obtained by Proposition~\ref{prp:four_spectral}(ii); in other words, $\done = 2$, $r_0 = 0$, $r_1,r_2 = 1$, and
\[
-J_\eta^2 = (b_1^\eta)^2 P_1^\eta + (b_2^\eta)^2 P_2^\eta,
\]
where $P_\jone^\eta = P_\jone^{J_\eta}$, $b_\jone^\eta = b_\jone^{J_\eta}$. In particular $(b_1^\eta)^2 (b_2^\eta)^2 = \det J_\eta = |\pf J_\eta|^2$, where $\pf \mu$ denotes the Pfaffian of $\mu \in \lie{so}_4$ (see, e.g., \cite[\S5.2]{bourbaki_algebre9}); moreover the preimage via $\eta \mapsto J_\eta$ of the cone $C$ coincides with the zero set of the quadratic polynomial $\eta \mapsto \pf J_\eta$.

Note that the polynomial $\eta \mapsto \pf J_\eta$, modulo change of sign and linear changes of variable, is an invariant of the isomorphism class of the Lie algebra $\lie{g}$: in fact, if $\omega_\eta$ is the alternating $2$-form on $\lie{g}/\snd$ defined by
\[
\omega_\eta(v+\snd,v'+\snd) = \eta([v,v'])
\]
for all $\eta \in \dusnd$ and $v,v'\in\lie{g}$, then $\eta \mapsto \omega_\eta$ is intrinsically defined (i.e., it does not depend on the choice of $\fst$ or $L$, or on any choice of coordinates) and
\[
\omega_\eta \wedge \omega_\eta = 2 (\pf J_\eta) \, \mathrm{Vol},
\]
where $\mathrm{Vol}$ is the volume form on $\lie{g}/\snd$ induced by the chosen identification of $\fst$ with $\R^4$. By the classification of quadratic forms, a suitable choice of linear coordinates $(\eta_1,\eta_2,\eta_3)$ on $\snd$ and of the orientation of $\fst$ then allows one to assume that $\pf J_\eta$ has one of the following forms:
\[
0, \quad \eta_1^2, \quad \eta_1 \eta_2, \quad \eta_1^2 + \eta_2^2, \quad \eta_1^2 + \eta_2^2 - \eta_3^2, \quad \eta_1^2 + \eta_2^2 + \eta_3^2.
\]
An inspection of the classification of the indecomposable $7$-dimensional $2$-step nilpotent real Lie algebras with $3$-dimensional center given by \cite{gong_classification_1998,kuzmich_graded_1999} shows that each of the above normal forms for $\pf J_\eta$ corresponds to exactly one isomorphism class, as summarized by the following table:

\begin{center}
\begin{tabular}{ccc}
name in \cite{gong_classification_1998} & name in \cite{kuzmich_graded_1999} & Pfaffian\\\hline
(37A)                  & m7\_2\_2  & $0$ \\
(37B)                  & m7\_2\_4  & $\eta_1 \eta_2$ \\
(37B\textsubscript{1}) & m7\_2\_4r & $\eta_1^2 + \eta_2^2$ \\
(37C)                  & m7\_2\_3  & $\eta_1^2$ \\
(37D)                  & m7\_2\_5  & $\eta_1^2 + \eta_2^2 - \eta_3^2$ \\
(37D\textsubscript{1}) & m7\_2\_5r & $\eta_1^2 + \eta_2^2 + \eta_3^2$
\end{tabular}
\end{center}

In the following we will refer to these isomorphism classes with the names given in \cite{gong_classification_1998}.
Note that the case (37A) coincides with the case $V \subseteq C$, which we have already discussed. In the case (37D\textsubscript{1}), on the other hand, $J_\eta$ is invertible for all $\eta \in \dusnd \setminus \{0\}$, hence $G$ is a M\'etivier group, which is covered by \cite{hebisch_multiplier_1993}.

About the remaining cases, we will show that (37B), (37B\textsubscript{1}) and (37C) can be treated analogously as the groups with $2$-dimensional second layer. The case (37D) is the most difficult and requires a special technique, which will be described in the next section. In the rest of the present section, we suppose that we are in one of the cases (37B), (37B\textsubscript{1}), (37C).

Let $\Omega_\cnnb = \{ \eta \in \dusnd \tc b_2^\eta < b_1^\eta/2 \}$ and $\Omega_\prnb = \{ \eta \in \dusnd \tc b_2^\eta > b_1^\eta/4 \}$. Note that $\Omega_\cnnb,\Omega_\prnb$ are an open cover of $\dusnd \setminus \{0\}$; the restriction to one of these open sets allow us to consider separately the components $\{\eta \tc |J_\eta^+| \, |J_\eta^-| = 0\}$ and $\{ \eta \tc |J_\eta^+| = |J_\eta^-|\}$ of the singular set $\dusnd \setminus \ddsnd$. We now show that, in each of $\Omega_\cnnb,\Omega_\prnb$, we can find suitable coordinates so that the hypothesis \eqref{eq:selfcontrolhypothesis} is satisfied.

\begin{lem}\label{lem:pcoords}
There exist linear coordinates $(\eta_1^\prnb,\eta_2^\prnb,\eta_3^\prnb)$ on $\dusnd$ such that, if $D_\prnb = (\eta_1^\prnb \partial_{\eta_1^\prnb},\eta_2^\prnb \partial_{\eta_2^\prnb},\eta_3^\prnb \partial_{\eta_3^\prnb})$ and $\tilde\Omega_\prnb = \Omega_\prnb \cap \{ \eta \tc \eta_1^\prnb \eta_2^\prnb \eta_3^\prnb \neq 0\}$, then
\[
b_1^\eta,b_2^\eta \in \SC_{\tilde\Omega_\prnb,D_\prnb}^\infty, \qquad P_1^\eta,P_2^\eta \in \BD_{\tilde\Omega_\prnb,D_\prnb}^\infty(1).
\]
\end{lem}
\begin{proof}
By identifying $\dusnd$ with its embedding $V$ in $\lie{so}_4$, the problem is reduced to choosing suitable coordinates on $V$. We know that $V_-$ and $V_+$ are proper subspaces  of $V$, and clearly $V_- \cap V_+ = 0$. Let $W$ be a linear complement of $V_- + V_+$ in $V$, and set $\tilde V_\pm = V_\pm + W$. Let $\pi_\pm : V \to \tilde V_\pm$ be the projection on the first component with respect to the decomposition $V = \tilde V_\pm \oplus V_\mp$.  Then, for all $\mu \in V$, $|\mu^\pm| = |\pi_\pm \mu|_\pm$ for some Euclidean norm $|\cdot|_\pm$ on $\tilde V_\pm$. Consequently, by Lemma~\ref{lem:schom_multi}, for any choice of coordinates $(\eta_1^\prnb,\eta_2^\prnb,\eta_3^\prnb)$ on $V$ compatible with the decomposition $V = V_+ \oplus V_- \oplus W$, the functions $\mu \mapsto |\mu^\pm|$ restricted to $V$ are in $\SC_{\tilde\Omega_\prnb,D_\prnb}^\infty$, and similarly the functions $\mu \mapsto \mu^\pm / |\mu^\pm|$ are in $\BD_{\tilde\Omega_\prnb,D_\prnb}^\infty(1)$. From the formulas given by Proposition~\ref{prp:four_spectral} and Leibniz' rule it is then clear that $\mu \mapsto P_1^\mu$ and $\mu \mapsto P_2^\mu$ restricted to $V$ are also in $\BD_{\tilde\Omega_\prnb,D_\prnb}^\infty(1)$ and moreover, by Lemma~\ref{lem:sc}(ii), $\mu \to b_1^\mu$ is in $\SC_{\tilde\Omega_\prnb,D_\prnb}^\infty$. Finally, $b_1^\mu b_2^\mu = ||\mu^+|^2-|\mu^-|^2|/2$ restricted to $\overline{\Omega_\prnb} \setminus \{0\}$ does not vanish, hence it is smooth,
and Lemma~\ref{lem:schom_unif} shows that $b_1^\mu b_2^\mu \in \SC_{\Omega_\prnb,D_\prnb}^\infty$, but then $b_2^\eta = (b_1^\eta)^{-1} (b_1^\eta b_2^\eta) \in \SC_{\tilde\Omega_\prnb,D_\prnb}$ by Lemma~\ref{lem:sc}(iii,iv).
\end{proof}

An inspection of the proof shows that the previous lemma would hold also in the case (37D). On the other hand, the assumption on the form of $\pf J_\eta$ is essential for the validity of the following lemma.

\begin{lem}\label{lem:ccoords}
There exist linear coordinates $(\eta_1^\cnnb,\eta_2^\cnnb,\eta_3^\cnnb)$ on $\dusnd$ such that, if $D_\cnnb = (\eta_1^\cnnb \partial_{\eta_1^\cnnb},\eta_2^\cnnb \partial_{\eta_2^\cnnb},\eta_3^\cnnb \partial_{\eta_3^\cnnb})$ and $\tilde\Omega_\cnnb = \Omega_\cnnb \cap \{\eta \tc \eta_1^\cnnb \eta_2^\cnnb \eta_3^\cnnb \neq 0\}$, then
\[
b_1^\eta,b_2^\eta \in \SC_{\tilde\Omega_\cnnb,D_\cnnb}^\infty, \qquad P_1^\eta,P_2^\eta \in \BD_{\tilde\Omega_\cnnb,D_\cnnb}^\infty(1).
\]
\end{lem}
\begin{proof}
Let $(\eta_1^\cnnb,\eta_2^\cnnb,\eta_3^\cnnb)$ be the coordinates on $\dusnd$ that bring the Pfaffian $\pf J_\eta$ in normal form as described above. It is then clear from the form of the Pfaffian and from Lemma~\ref{lem:schom_multi} that $\eta \mapsto b_1^\eta b_2^\eta$ is in $\SC_{\tilde\Omega_\cnnb,D_\cnnb}^\infty$. On the other hand, the functions $\mu \mapsto \mu^\pm$ do not vanish on $\overline{\Omega_\cnnb} \setminus \{0\}$, hence the functions $\eta \mapsto P_\jone^\eta$ and $\eta \mapsto b_1^\eta$ are smooth there, and Lemma~\ref{lem:schom_unif} show that $b_1^\eta \in \SC_{\Omega_\cnnb,D_\cnnb}^\infty$ and $P_1^\eta,P_2^\eta \in \BD_{\Omega_\cnnb,D_\cnnb}^\infty(1)$. Consequently $b_2^\eta = (b_1^\eta)^{-1} (b_1^\eta b_2^\eta) \in \SC_{\tilde\Omega_\cnnb,D_\cnnb}^\infty$ by Lemma~\ref{lem:sc}(iii,iv).
\end{proof}

Let $(u_1^\prnb,u_2^\prnb,u_3^\prnb)$ the coordinates on $\snd$ dual to the coordinates $(\eta_1^\prnb,\eta_2^\prnb,\eta_3^\prnb)$ given by Lemma~\ref{lem:pcoords}; analogously, let  $(u_1^\cnnb,u_2^\cnnb,u_3^\cnnb)$ be the coordinates on $\snd$ dual to the coordinates $(\eta_1^\cnnb,\eta_2^\cnnb,\eta_3^\cnnb)$ given by Lemma~\ref{lem:ccoords}. Let $\zeta_\cnnb,\zeta_\prnb : \dusnd \setminus \{0\} \to \R$ be a partition of unity subordinate to the open cover $\Omega_\cnnb,\Omega_\prnb$ and made of smooth functions homogeneous of degree $0$.
A repetition of the proof of Proposition~\ref{prp:2dc_secondweightedestimate} then yields the following estimates.

\begin{prp}
For all compact sets $K \subseteq \R$, for all Borel functions $F : \R \to \C$ supported in $K$, and for all $\alpha \in \leftclosedint 0,1/2 \rightopenint^3$,
\begin{gather*}
\int_G \left| \Kern_{F(L) \, \zeta_\prnb(\vecU)}(x,u) \prod_{\jtwo=1}^3 (1+|u_\jtwo^\prnb|)^{\alpha_\jtwo} \right|^2 \,dx\,du \leq C_{K,\alpha} \|F\|_{W_2^{|\alpha|}}^2,\\
\int_G \left| \Kern_{F(L) \, \zeta_\cnnb(\vecU)}(x,u) \prod_{\jtwo=1}^3 (1+|u_\jtwo^\cnnb|)^{\alpha_\jtwo} \right|^2 \,dx\,du \leq C_{K,\alpha} \|F\|_{W_2^{|\alpha|}}^2.
\end{gather*}
\end{prp}

As in Proposition~\ref{prp:2dc_l1estimate} these estimates can be combined at the level of $L^1$, so to obtain the following result, which completes the proof of Proposition~\ref{prp:l1low}, except for the case (37D).

\begin{prp}\label{prp:7d_l1estimate}
Let $K \subseteq \R$ be compact. For all Borel functions $F : \R \to \C$ such that $\supp F \subseteq K$, and for all $\alpha,\beta \in \R$ such that $\beta \geq 3/2$ and $\beta > \alpha + 7/2$,
\[
\| (1+|\cdot|_G)^\alpha \, \Kern_{F(L)} \|_1 \leq C_{K,\alpha,\beta} \| F \|_{W_2^\beta}.
\]
\end{prp}

\section{A particular group with $3$-dimensional second layer}\label{section:g743}

In this section we assume that we are in the case (37D) according to the classification described in \S\ref{section:7dim}. Therefore we can choose orthogonal coordinates $(x_1,x_2,x_3,x_4)$ on $\fst$ and coordinates $(\eta_1,\eta_2,\eta_3)$ on $\dusnd$ such that
\begin{equation}\label{eq:pfaffiancondition}
\pf J_\eta = \eta_1^2 + \eta_2^2 - \eta_3^2.
\end{equation}
Let us fix an inner product on $\dusnd$ such that $(\eta_1,\eta_2,\eta_3)$ are orthogonal coordinates. We may suppose that the homogeneous norm $|\cdot|_G$ on $G$ is defined by \eqref{eq:homogeneousnorm} where the norms on $\fst$ and $\snd$ are induced by the chosen inner products.

In contrast with the result of \S\ref{section:7dim}, in this case we are not able to find coordinates on $\dusnd$ for which the hypothesis \eqref{eq:selfcontrolhypothesis} is satisfied in a neighborhood of the cone $\{ \eta \tc \pf J_\eta = 0\}$; a more refined decomposition will then be used, involving an infinite number of systems of coordinates. An additional ingredient that will be exploited is a special extra weight on the first layer, given by an adaptation of the technique of \cite{hebisch_multiplier_1993,hebisch_multiplier_1995} and \cite[\S3]{martini_joint_2012}, and by the following estimates.

\begin{lem}\label{lem:g743_hzestimate}
There exists a continuous function $w : \fst \to \leftclosedint 0,\infty \rightopenint$ such that:
\begin{itemize}
\item[(i)] for all $\eta \in \dusnd$ and $x \in \fst$,
\[
|J_\eta x| \geq |\eta| \, w(x);
\]
\item[(ii)] if $\alpha,\gamma \in \leftclosedint 0,\infty \rightopenint$ and $\min\{\gamma,1\} + \alpha > \dim \fst$, then
\[
\int_{\fst} (1+|x|)^{-\alpha} \,(1+w(x))^{-\gamma} \,dx < \infty.
\]
\end{itemize}
\end{lem}
\begin{proof}
Suppose first that, in the chosen coordinates,
\begin{equation}\label{eq:g743_specialform}
J_\eta = \begin{pmatrix}
0               & 0              & -\eta_1 - \eta_3 & -\eta_2         \\
0               & 0              & -\eta_2          & \eta_1 - \eta_3 \\
\eta_1 + \eta_3 & \eta_2         & 0                & 0               \\
\eta_2          & -\eta_1+\eta_3 & 0                & 0
\end{pmatrix}.
\end{equation}
One may check that the previous formula indeed defines a $2$-step stratified structure on $\R^4 \times \R^3$ and that \eqref{eq:pfaffiancondition} holds. Proposition~\ref{prp:four_spectral} then gives that
\[
P_1^\eta-P_2^\eta = \frac{\sgn \eta_3}{\sqrt{\eta_1^2 + \eta_2^2}} \begin{pmatrix}
\eta_1 & \eta_2  & 0      & 0      \\
\eta_2 & -\eta_1 & 0      & 0      \\
0      & 0       & \eta_1 & \eta_2 \\
0      & 0       & \eta_2 & -\eta_1
\end{pmatrix}
\]
and
\[
b_1^\eta = \sqrt{\eta_1^2 + \eta_2^2} + |\eta_3|, \qquad b_2^\eta = \left|\sqrt{\eta_1^2+\eta_2^2}-|\eta_3|\right|.
\]
Define the function $w : \fst \to \leftclosedint 0,\infty \rightopenint$ by
\[
w(x) = \sqrt{|x|^2 - \sqrt{(x_1^2-x_2^2+x_3^2-x_4^2)^2 + (2x_1 x_2 + 2x_3 x_4)^2}}.
\]
For all $\eta \in \dusnd$ and $x \in \fst$,
\[\begin{split}
|J_\eta x|^2 &= \langle -J_\eta^2 x, x \rangle = (b_1^\eta)^2 \langle P_1^\eta x,x \rangle + (b_2^\eta)^2 \langle P_2^\eta x, x \rangle \\
&= \frac{(b_1^\eta)^2 + (b_2^\eta)^2}{2} \langle (P_1^\eta + P_2^\eta) x, x \rangle + \frac{(b_1^\eta)^2 - (b_2^\eta)^2}{2} \langle (P_1^\eta - P_2^\eta) x, x \rangle.
\end{split}\]
Since $(b_1^\eta)^2 + (b_2^\eta)^2 = 2|\eta|^2$, $(b_1^\eta)^2 - (b_2^\eta)^2 = 4 |\eta_3| \sqrt{\eta_1^2 + \eta_2^2}$, $P_1^\eta + P_2^\eta = 1$, 
we deduce that
\[
|J_\eta x|^2 = |\eta|^2 |x|^2 - 2 \eta_3 (v_1 \eta_1 + v_2 \eta_2).
\]
where
\[
v_1 = x_1^2 - x_2^2 + x_3^2 - x_4^2, \qquad v_2 = 2x_1 x_2 + 2x_3 x_4.
\]
The Cauchy-Schwarz inequality gives us that
\[
|2 \eta_3 (v_1 \eta_1 + v_2 \eta_2)| \leq 2 |\eta_3| \sqrt{\eta_1^2+\eta_2^2} \sqrt{v_1^2+v_2^2} \leq |\eta|^2 \sqrt{v_1^2 + v_2^2}
\]
and since $w(x)^2 = |x|^2 - \sqrt{v_1^2 + v_2^2}$, part (i) follows.

As for part (ii), choose $\gamma' \in \leftclosedint 0, 1 \rightopenint$ such that $\gamma' \leq \gamma$ and $\gamma' + \alpha > \dim \fst$. Then
\[
\int_{\fst} (1+|x|)^{-\alpha} \,(1+w(x))^{-\gamma} \,dx \leq \int_{\fst} (1+|x|)^{-\alpha} \,(1+w(x))^{-\gamma'} \,dx,
\]
and moreover, by the properties of $w$,
\[
1+w(x) = 1 + |x| \,w(x/|x|) \geq (1+|x|) \, w(x/|x|),
\]
therefore
\[
\int_{\fst} (1+|x|)^{-\alpha} \,(1+w(x))^{-\gamma'} \,dx \leq \int_{\fst} (1+|x|)^{-\alpha-\gamma'} w(x/|x|)^{-\gamma'} \,dx.
\]
If the integral in the right-hand side is performed in polar coordinates, then the convergence of the radial part follows from the assumption $\gamma' + \alpha > \dim\fst$, and we are left with the angular part
\[
\int_{S^3} w(\omega)^{-\gamma'} \,d\omega,
\]
Note now that, for all $\omega \in S^3$,
\[
w(\omega)^2 \sim 1-(\omega_1^2-\omega_2^2+\omega_3^2-\omega_4^2)^2-(2\omega_1\omega_2+2\omega_3\omega_4)^2 = (2\omega_1 \omega_4 -2\omega_2 \omega_3)^2,
\]
and it is easily checked that $\omega \mapsto \omega_1\omega_4 - \omega_2\omega_3$ vanishes of first-order on $S^3$ (its gradient as a function on $\R^4$ is never normal to $S^3$ on the zero set of the function), hence the integral on $S^3$ converges because $\gamma'<1$.

\smallskip

We have thus completed the proof in the particular case where \eqref{eq:g743_specialform} holds. Note now that (i) can be equivalently rewritten as
\[
\sup_{x' \in \fst \setminus \{0\}} \frac{|\eta([x,x'])|}{|x'|} \geq |\eta| \, w(x).
\]
If we replace the norms on $\fst$ and $\dusnd$ with equivalent norms, then the previous inequality still holds, modulo multiplying $w$ by a suitable nonzero constant; these modifications clearly preserve also the validity of (ii). Since by the aforementioned classification result of \cite{gong_classification_1998,kuzmich_graded_1999} there is only one indecomposable $2$-step stratified group (up to isomorphism) such that $d = 7$ and $\dtwo = 3$ and \eqref{eq:pfaffiancondition} holds in suitable coordinates, the conclusion follows in the general case.
\end{proof}

Let $\Omega_\cnnb = \{ \eta \in \dusnd \tc b_1^\eta b_2^\eta < \smct |\eta|^2 \}$ and $\Omega_\prnb = \{ \eta \in \dusnd \tc b_1^\eta b_2^\eta > \smct |\eta|^2/2 \}$, where $\smct \in \leftopenint 0,1/2 \rightopenint$ is sufficiently small so that
\begin{equation}\label{eq:g743_cnnbcontainment}
\Omega_\cnnb \subseteq \{\eta \tc b_2^\eta < b_1^\eta/2 \}.
\end{equation}
Let $\zeta_\cnnb,\zeta_\prnb : \dusnd \setminus \{0\} \to \R$ be a smooth homogeneous partition of unity subordinate to the open cover $\Omega_\cnnb,\Omega_\prnb$. Since $\Omega_\cnnb$ does not intersect the plane $\{\eta \tc \eta_3=0\}$, $\zeta_\cnnb$ decomposes uniquely as $\zeta_+ + \zeta_-$, where $\zeta_\pm : \dusnd \setminus \{0\} \to \R$ is smooth and supported in $\Omega_\cnnb \cap \{ \eta \tc \pm \eta_3 > 0 \}$.

We consider first the region $\Omega_\prnb$, that is the region far from the cone $\{ \eta \tc \pf J_\eta = 0\}$. This is the ``easy part'' to be considered, since a single system of coordinates is sufficient. In fact, by repeating the proof of Lemma~\ref{lem:pcoords}, we obtain immediately the following result.

\begin{lem}\label{lem:g743_pcoords}
There exist coordinates $(\eta_1^\prnb,\eta_2^\prnb,\eta_3^\prnb)$ on $\dusnd$ such that $\max_\jtwo|\eta_\jtwo^\prnb| \leq |\eta|$ and, if $D_\prnb = (\eta_1^\prnb \partial_{\eta_1^\prnb},\eta_2^\prnb \partial_{\eta_2^\prnb},\eta_3^\prnb \partial_{\eta_3^\prnb})$ and $\tilde\Omega_\prnb = \Omega_\prnb \cap \{ \eta \tc \eta_1^\prnb \eta_2^\prnb \eta_3^\prnb \neq 0\}$, then
\[
b_1^\eta,b_2^\eta \in \SC_{\tilde\Omega_\prnb,D_\prnb}^\infty, \qquad P_1^\eta,P_2^\eta \in \BD_{\tilde\Omega_\prnb,D_\prnb}^\infty(1).
\]
\end{lem}

Let $(\eta_1^\prnb,\eta_2^\prnb,\eta_3^\prnb)$ be the coordinates on $\dusnd$ given by Lemma~\ref{lem:g743_pcoords}, and let $(u_1^\prnb,u_2^\prnb,u_3^\prnb)$ be the dual coordinates on $\snd$. Let $w : \fst \to \leftclosedint 0,\infty \rightopenint$ be the function given by Lemma~\ref{lem:g743_hzestimate}. Let $\chi \in C^\infty_\cnnb(\leftopenint 0,\infty \rightopenint)$ be such that $\supp \chi \subseteq \leftclosedint 1/2,2 \rightclosedint$ and $\sum_{n \in \Z} \chi(2^n t) = 1$ for all $t \in \leftopenint 0,\infty \rightopenint$.

\begin{prp}\label{prp:g743_firstweightedestimate}
Let $K \subseteq \R$ be compact. For all smooth $F : \R \to \C$ such that $\supp F \subseteq K$,
for all $\radp \in \leftopenint 0,\infty \rightopenint$, $\crdp \in \leftopenint 0,\infty \rightopenint^3$, if $\zeta_{\prnb,\radp,\crdp} : \dusnd \to \C$ is defined by
\[
\zeta_{\prnb,\radp,\crdp}(\eta) = \zeta_\prnb(\eta) \, \chi(|\eta|/\radp) \prod_{\jtwo = 1}^3 \chi(|\eta^\prnb_\jtwo|/(\crdp_\jtwo|\eta|)),
\]
then, for all $\alpha \in \leftclosedint 0,\infty \rightopenint^3$ and $\theta \in \leftclosedint 0,\infty \rightopenint$,
\begin{multline}\label{eq:g743_firstweightedestimate}
\int_{G} \Bigl| \Kern_{F(L) \, \zeta_{\prnb,\radp,\crdp}(\vecU)}(x,u) \, (1+w(x))^\theta \prod_{\jtwo=1}^3 (1+|u^\prnb_\jtwo|)^{\alpha_\jtwo} \Bigr|^2 \,dx \,du \\
 \leq C_{K,\alpha,\theta} \, \|F\|_{W_2^{|\alpha|}} \, \radp^{3-2|\alpha|-2\theta} \prod_{\jtwo=1}^3 \crdp_\jtwo^{1-2\alpha_\jtwo} .
\end{multline}
\end{prp}
\begin{proof}
As in the proof of Proposition~\ref{prp:2dc_secondweightedestimate}, we may assume that $\crdp_1,\crdp_2,\crdp_3 \leq 2$ and that $\radp \leq C_K$, otherwise $F(L)\, \zeta_{\prnb,\radp,\crdp}(\vecU) = 0$.

From Lemma~\ref{lem:g743_pcoords} and Proposition~\ref{prp:cutoff_partiall2norm}, we get immediately that, for all $\alpha \in \N^3$,
\begin{multline}\label{eq:g743_p_1}
\int_{G} \Bigl| |u^\prnb_1|^{\alpha_1} |u^\prnb_2|^{\alpha_2} |u^\prnb_3|^{\alpha_3}\, \Kern_{F(L) \, \zeta_{\prnb,\radp,\crdp}(\vecU)}(x,u) \Bigr|^2 \,dx \,du \\
\leq C_{K,\alpha} \|F\|_{W_2^{|\alpha|}}^2 \prod_{\jtwo=1}^3 (\radp\crdp_\jtwo)^{1-2\alpha_\jtwo}.
\end{multline}
The previous inequality extends to all $\alpha \in \leftclosedint 0,\infty \rightopenint^3$ by interpolation.

On the other hand, by Lemma~\ref{lem:g743_hzestimate}, $|J_\eta x| \geq |\eta| \, w(x)$. Hence, by \cite[Proposition 3.5]{martini_joint_2012} and Corollary~\ref{cor:plancherelmeasure}, for all $\theta \in \leftclosedint 0,\infty \rightopenint$,
\begin{multline*}
\int_{G} \Bigl| w(x)^\theta \, \Kern_{F(L) \, \zeta_{\prnb,\radp,\crdp}(\vecU)}(x,u) \Bigr|^2 \,dx \,du \\
\leq C_{K,\theta}
\int_{\ddsnd} \sum_{n \in \N^2} |F(b_1^\eta(2n_1+1) + b_2^\eta (2n_2+1))|^2 \,|\zeta_{\prnb,\radp,\crdp}(\eta)|^2 \, |\eta|^{-2\theta} \, b_1^\eta \,b_2^\eta \,d\eta.
\end{multline*}
Analogously as in the proof of Proposition~\ref{prp:cutoff_partiall2norm}, by passing to polar coordinates and rescaling, we easily obtain that
\begin{equation}\label{eq:g743_p_2}
\int_{G} \Bigl| w(x)^\theta \Kern_{F(L) \, \zeta_{\prnb,\radp,\crdp}(\vecU)}(x,u) \Bigr|^2 \,dx \,du \leq C_{K,\theta} \|F\|_{W_2^{0}}^2 \, \radp^{3-2\theta} \crdp_1 \crdp_2 \crdp_3.
\end{equation}

By interpolating \eqref{eq:g743_p_1} and \eqref{eq:g743_p_2}, we obtain that, for all $\alpha \in \leftclosedint 0,\infty \rightopenint^3$ and $\theta \in \leftclosedint 0,\infty \rightopenint$,
\begin{multline*}
\int_{G} \Bigl| w(x)^\theta \, |u^\prnb_1|^{\alpha_1} |u^\prnb_2|^{\alpha_2} |u^\prnb_3|^{\alpha_3}\, \Kern_{F(L) \, \zeta_{\prnb,\radp,\crdp}(\vecU)}(x,u) \Bigr|^2 \,dx \,du \\
\leq C_{K,\alpha,\theta} \|F\|_{W_2^{|\alpha|}}^2 \, \radp^{-2\theta} \prod_{\jtwo=1}^3 (\radp \crdp_\jtwo)^{1-2\alpha_\jtwo} .
\end{multline*}
The conclusion follows by combining this inequality with the corresponding ones where $\theta$ and/or some of the components of $\alpha$ are replaced by $0$.
\end{proof}

If $\alpha_1,\alpha_2,\alpha_3,\theta$ are sufficiently small, then the exponents of $\radp,\crdp_1,\crdp_2,\crdp_3$ in \eqref{eq:g743_firstweightedestimate} are positive; hence, as in the proof of Proposition~\ref{prp:2dc_secondweightedestimate}, the estimates given by the previous proposition can be summed via a dyadic decomposition, in order to obtain the following result.

\begin{cor}\label{cor:g743_firstweightedestimate}
Let $K \subseteq \R$ be compact. For all smooth $F : \R \to \C$ such that $\supp F \subseteq K$,
for all $\alpha \in \leftclosedint 0,1/2 \rightopenint^3$ and $\theta \in \leftclosedint 0,3/2-|\alpha|\rightopenint$,
\[
\int_{G} \Bigl| \Kern_{F(L) \, \zeta_\prnb(\vecU)}(x,u)  \, (1+w(x))^\theta \prod_{\jtwo=1}^3 (1+|u_\jtwo^\prnb|)^{\alpha_\jtwo} \Bigr|^2 \,dx \,du \leq C_{K,\alpha,\theta} \|F\|_{W_2^{|\alpha|}}^2.
\]
\end{cor}

H\"older's inequality then yields the following $L^1$-estimate.

\begin{cor}\label{cor:g743_firstweightedestimate_final}
Let $K \subseteq \R$ be compact. For all smooth $F : \R \to \C$ such that $\supp F \subseteq K$, for all $\alpha,\beta \in \R$ such that $\beta \geq 0$, $2\beta > \alpha+9/2$, $\alpha < -5/2$,
\[
\| (1+|\cdot|_G)^{\alpha} \, \Kern_{F(L) \, \zeta_\prnb(\vecU)}\|_1 \leq C_{K,\alpha,\beta} \|F\|_{W_2^{\beta}}.
\]
\end{cor}
\begin{proof}
Under our hypothesis, we can choose $\alpha_1 \in \leftopenint 3/2, -\alpha-1-2(1-\beta)_+ \rightopenint$. Hence, if $\alpha_2 = (-\alpha-\alpha_1)/2$, then $-\alpha = \alpha_1 + 2\alpha_2$ and $\alpha_2 > 1/2 + (1-\beta)_+$, and therefore we can choose $s \in \leftopenint 3/2-\alpha_2,\beta\rightclosedint \cap \leftclosedint 0,1\rightopenint$. Consequently, by H\"older's inequality and Corollary~\ref{cor:g743_firstweightedestimate},
\begin{multline*}
\| (1+|\cdot|_G)^{\alpha} \, \Kern_{F(L) \, \zeta_\prnb(\vecU)}\|_1^2 \\
 \leq C_{K,\alpha,\beta} \|F\|_{W_2^{s}}^2 \int_G (1+|x|)^{-2\alpha_1} \,(1+w(x))^{-1} \prod_{\jtwo=1}^3 (1+|u_\jtwo|)^{-2(\alpha_2+s)/3} \,dx\,du.
\end{multline*}
Since $2\alpha_1 + 1 > 4$, and $2(\alpha_2+s) > 3$, the last integral is finite by Lemma~\ref{lem:g743_hzestimate}, and the conclusion follows because $s \leq \beta$.
\end{proof}

Let us consider now the ``hard part'', that is the region $\Omega_\cnnb$ near the cone $\{ \eta \tc \pf J_\eta = 0\}$. This region will be decomposed into an infinite number of pieces, for each of which a specific system of coordinates will be used. The decomposition can be described in two steps:
\begin{itemize}
\item first decomposition: we decompose $\Omega_\cnnb$ in ``truncated conic shells'' where the distance from the origin and the distance from the cone are approximately constant, i.e., $|\eta| \sim \radp$ and $b_1^\eta b_2^\eta / |\eta|^2 \sim \crdp$ for some (small, dyadic) parameters $\radp,\crdp \in \leftopenint 0,\infty \rightopenint$; each of these shells is invariant by rotations around the axis $\{ \eta \tc \eta_1 = \eta_2 = 0\}$ of the cone;
\item second decomposition: each shell given by the first decomposition is further decomposed into ``sectors'', according to an angular parameter (i.e., the argument of $(\eta_1,\eta_2)$), with angular width $\sim \crdp^{1/2}$; as it turns out, in each of these sectors an orthonormal system of coordinates (with axes approximately given by the radial direction, the normal to the cone, and the tangent to the cone parallel to the plane $\{\eta \tc \eta_3 = 0\}$) can be chosen so to satisfy the hypothesis \eqref{eq:selfcontrolhypothesis}.
\end{itemize}
Due to the fact that this decomposition must be achieved via a smooth partition of unity, and that the estimates to be obtained (which depend on the derivatives of the components of the partition of unity) must be uniform from piece to piece, the details of the decomposition are slightly technical. Some help is given by the rotational invariance of the cone; note however that $b_1^\eta$ and $b_2^\eta$ need not be invariant by rotations around the axis of the cone.

Recall that $\dusnd$ is identified with $\R^3 = \R^2 \times \R$ via the coordinates $(\eta_1,\eta_2,\eta_3)$. For all $\crdp \in \leftopenint 0,1 \rightclosedint$, let $I_{\epsilon}$ and $(\chi_{\epsilon,v})_{v \in I_\epsilon}$ denote the subset of $S^1$ and the homogeneous partition of unity of $\R^2 \setminus \{0\}$ given by Lemma~\ref{lem:hompartuni} corresponding to the thinness parameter $\thinp = \smctb \crdp^{1/2}$, where $\smctb \in \leftopenint 0,1/4 \rightopenint$ is a small constant to be fixed later; set moreover $Y_\crdp = I_\epsilon \times \{-1,+1\}$, and define, for all $q = (v,\pm 1) \in Y_\crdp$,
\[
\tilde\chi_{\crdp,q}(\eta) = \chi_{\epsilon,v}(\eta_1,\eta_2).
\]

For all $v \in S^1$, let $v^\perp \in S^1$ be such that $v,v^\perp$ is a positive orthonormal basis of $\R^2$. 
For all $q = (v,\pm 1) \in S^1 \times \{-1,+1\}$, let $(\eta_1^q,\eta_2^q,\eta_3^q)$ be the orthonormal coordinates on $\dusnd$ corresponding to the basis $(v,\pm 1)/\sqrt{2},(v^\perp,0),(v,\mp 1)/\sqrt{2}$. It is then easily seen that
\begin{equation}\label{eq:g743_pfaffian2}
\pf J_\eta = 2\eta_1^{q} \eta_3^{q} + (\eta_2^{q})^2.
\end{equation}
Set $D_{q} = (\eta_1^{q} \partial_{\eta_1^{q}},\eta_2^{q} \partial_{\eta_2^{q}}, \eta_3^{q} \partial_{\eta_3^{q}})$ and $V_q = \{ \eta \tc 2 (\eta_2^q)^2 < |\pf J_\eta| \}$.

Further, for all $\radp \in \leftopenint 0,\infty \rightopenint$, $\crdp \in \leftopenint 0,1 \rightclosedint$, $q = (v,\pm 1) \in Y_\crdp$, set
\begin{gather*}
\zeta_{\cnnb,\crdp,q}(\eta) 
= \zeta_\pm(\eta) \, \chi( b_1^\eta b_2^\eta / (|\eta|^2 \crdp) ) \, \tilde\chi_{\crdp,q}(\eta), \\
\zeta_{\cnnb,\radp,\crdp,q}(\eta) = \zeta_{\cnnb,\crdp,q}(\eta) \, \chi(|\eta|/\radp).
\end{gather*}
Each cutoff $\zeta_{\cnnb,\radp,\crdp,q}$ corresponds to one of the sectors given by the second decomposition, and $(\eta_1^q,\eta_2^q,\eta_3^q)$ are the coordinates meant to be used there. The following lemma collects the estimates on the derivatives of these cutoffs and on the sizes of their supports, together with the estimates on the derivatives of $b_1^\eta,b_2^\eta,P_1^\eta,P_2^\eta$, which are needed to apply the machinery of \S\ref{section:kernelderivatives}.

\begin{lem}\label{lem:g743_uniformestimates}
For all $A \in \N$ there exists $\kappa \in \leftclosedint 1,\infty \rightopenint$ such that, for all $\radp \in \leftopenint 0,\infty\rightopenint$, $\crdp \in \leftopenint 0,1 \rightclosedint$, $q \in Y_\crdp$, the following holds:
\begin{itemize}
\item[(i)] $\|b_1^\eta b_2^\eta \|_{\SC_{V_q,D_q}^A} \leq \kappa$,
\item[(ii)] $\|b_1^\eta\|_{\SC_{\Omega_\cnnb,D_q}^A},\|P_1^\eta\|_{\BD_{\Omega_\cnnb,D_q}^A(1)},\|P_2^\eta\|_{\BD_{\Omega_\cnnb,D_q}^A(1)} \leq \kappa$,
\item[(iii)] $\|b_2^\eta\|_{\SC_{\Omega_\cnnb \cap V_q,D_{q}}^A} \leq \kappa$,
\item[(iv)] $\supp \zeta_{\cnnb,\crdp,q} \subseteq \Omega_\cnnb \cap V_q$,
\item[(v)] $\|\zeta_{\cnnb,\crdp,q}\|_{\BD_{\Omega_\cnnb \cap V_q,D_q}^A(1)} \leq \kappa$,
\item[(vi)] $|\eta_1^q|/\radp, \, |\eta_2^q|/(\radp \crdp^{1/2}), \, |\eta_3^q|/(\radp \crdp) \in \leftclosedint \kappa^{-1}, \kappa \rightclosedint$ for all $\eta \in \supp \zeta_{\cnnb,\radp,\crdp,q}$.
\end{itemize}
\end{lem}
\begin{proof}
By \eqref{eq:g743_pfaffian2}, the only nonzero iterated $D_p$-derivatives of $\pf J_\eta$ are constant multiples of $\eta_1^q \eta_3^q$ or of $(\eta_2^q)^2$. Since
\begin{equation}\label{eq:g743_firstetaestimates}
(\eta_2^q)^2 \lesssim |\pf J_\eta| \quad\text{and}\quad |\eta_1^q \eta_3^q| \sim |\pf J_\eta| \qquad\text{for } \eta \in V_q,
\end{equation}
it is clear that $\pf J_\eta$ and $b_1^\eta b_2^\eta = |\pf J_\eta|$ are in $\SC_{V_q,D_q}^A$; the $\SC$-norm does not depend on $q$ because $\pf J_\eta$ has the same form \eqref{eq:g743_pfaffian2} in all coordinates $(\eta_1^q,\eta_2^q,\eta_3^q)$, and part (i) is proved.

Note that, by Proposition~\ref{prp:four_spectral}, $b_1^\eta,b_2^\eta,P_1^\eta,P_2^\eta$ are smooth on $\{\eta \tc b_2^\eta < b_1^\eta \}$ and $b_1^\eta$ does not vanish there. Part (ii) then follows by \eqref{eq:g743_cnnbcontainment} and Lemma~\ref{lem:schom_unif}. Since $b_2^\eta = (b_1^\eta)^{-1} (b_1^\eta b_2^\eta)$, part (iii) follows from parts (i) and (ii) and from Lemma~\ref{lem:sc}.

Let $q = (v,\pm 1)$. Clearly $\supp \zeta_{\cnnb,\crdp,q} \subseteq \supp \zeta_\pm \subseteq \Omega_\cnnb$.
Moreover, by Lemma~\ref{lem:hompartuni}(ii),
\begin{equation}\label{eq:g743_eta2firstestimate}
|\eta_2^q| \sim \smctb \crdp^{1/2} |(\eta_1,\eta_2)| \qquad\text{for } \eta \in \supp \tilde\chi_{\crdp,q}
\end{equation}
and also
\begin{equation}\label{eq:g743_eta12estimate}
|(\eta_1,\eta_2)|^2 = (|\eta|^2 + \pf J_\eta)/2 \sim |\eta|^2  \qquad\text{for } \eta \in \Omega_\cnnb.
\end{equation}
Therefore
\begin{equation}\label{eq:g743_eta2estimate}
|\eta_2^q|^2 \sim \smctb^2 \crdp |\eta|^2 \sim \smctb^2 b_1^\eta b_2^\eta \qquad\text{for } \eta \in \supp \zeta_{\cnnb,\crdp,q};
\end{equation}
by choosing $\smctb$ sufficiently small, we obtain that $\supp \zeta_{\cnnb,\crdp,q} \subseteq V_q$, and part (iv) is proved.

From part (i) and Lemmata~\ref{lem:schom_multi} and \ref{lem:sc} we deduce that $b_1^\eta b_2^\eta/|\eta|^2 \in \SC_{V_q,D_q}^A$, therefore by \eqref{eq:sc_scalar} and Lemma~\ref{lem:sccomp2}(i) also $\eta \mapsto \chi( b_1^\eta b_2^\eta / (|\eta|^2 \crdp) )$ is in $\BD_{V_q,D_q}^A(1)$, with norm uniformly bounded in $q,\crdp$. Moreover, by Lemma~\ref{lem:hompartuni}(iii) and \eqref{eq:g743_eta2firstestimate} and \eqref{eq:g743_eta12estimate}, for all $\alpha \in \N^3$ and $\eta \in \Omega_\cnnb$,
\[
|\partial_{\eta^q}^\alpha \tilde\chi_{q,\crdp}(\eta)| \leq C_\alpha |\eta|^{-\alpha_1-\alpha_3} |\eta^q_2|^{-\alpha_2};
\]
hence, as in the proof of Lemma~\ref{lem:hompartuni}(iv), one sees that $\tilde\chi_{\crdp,q}$ is in $\BD_{\Omega_c,D_q}^A(1)$ with norm uniformly bounded in $q,\crdp$. Further $\chi_\pm \in \BD_{\dusnd \setminus \{0\},D_q}^A(1)$ with norm uniformly bounded in $q$ by homogeneity and Lemma~\ref{lem:schom_unif}. Part (v) then follows by Leibniz' rule.

Finally, by Lemma~\ref{lem:hompartuni}(ii) and the fact that $\supp \zeta_\pm \subseteq \{\eta \tc \pm \eta_3 > 0\}$,
\[
\eta_1^q = (\langle (\eta_1,\eta_2), v \rangle \pm \eta_3 )/\sqrt{2} \sim |(\eta_1,\eta_2)| + |\eta_3| \sim \radp \qquad\text{for } \eta \in \supp\zeta_{\cnnb,\radp,\crdp,q}, 
\]
and moreover, by \eqref{eq:g743_eta2estimate}, it is clear that $|\eta_2^q| \sim \crdp^{1/2} \radp$ for $\eta \in \supp\zeta_{\cnnb,\radp,\crdp,q}$. On the other hand, by \eqref{eq:g743_firstetaestimates} and part (iv), $|\eta_1^q \eta_3^q| \sim b_1^\eta b_2^\eta \sim \crdp \radp^2$ for $\eta \in \supp\zeta_{\cnnb,\radp,\crdp,q}$; consequently $|\eta_3^q| = |\eta_1^q \eta_3^q| / |\eta_1^q| \sim \crdp\radp$ for $\eta \in \supp\zeta_{\cnnb,\radp,\crdp,q}$, and part (vi) is proved.
\end{proof}

Denote by $(u_1^q,u_2^q,u_3^q)$ the system of coordinates on $\snd$ dual to $(\eta_1^q,\eta_2^q,\eta_3^q)$ on $\dusnd$. A repetition of the proof of Proposition~\ref{prp:g743_firstweightedestimate}, exploiting Lemma~\ref{lem:g743_uniformestimates} in place of Lemma~\ref{lem:g743_pcoords}, yields the following estimate.

\begin{prp}\label{prp:g743_coneweightedestimate}
Let $K \subseteq \R$ be compact. For all smooth $F : \R \to \C$ such that $\supp F \subseteq K$,
for all $\radp \in \leftopenint 0,\infty \rightopenint$, $\crdp \in \leftopenint 0,1 \rightopenint$, $q \in Y_\crdp$,
$\alpha \in \leftclosedint 0,\infty \rightopenint^3$ and $\theta \in \leftclosedint 0,\infty \rightopenint$,
\begin{multline*}
\int_{G} \Bigl|\Kern_{F(L) \, \zeta_{\cnnb,\radp,\crdp,q}(\vecU)}(x,u) \, (1+w(x))^\theta \prod_{\jtwo=1}^3 (1+|u^q_\jtwo|)^{\alpha_\jtwo} \Bigr|^2 \,dx \,du \\
\leq C_{K,\alpha,\theta} \, \radp^{3-2|\alpha|-2\theta} \crdp^{3/2-\alpha_2-2\alpha_3} \|F\|_{W_2^{|\alpha|}}^2.
\end{multline*}
\end{prp}

Unfortunately we cannot sum directly the estimates given by the previous proposition, since the weight changes from piece to piece. In order to avoid this problem, we must first apply H\"older's inequality in order to obtain $L^1$-estimates with a weight independent of the piece. The next result estimates the contribution given by each of the shells corresponding to the first decomposition.

\begin{cor}\label{cor:g743_coneweightedestimate2}
Let $K \subseteq \R$ be compact. For all smooth $F : \R \to \C$ such that $\supp F \subseteq K$,
for all $\radp \in \leftopenint 0,\infty \rightopenint$ and $\crdp \in \leftopenint 0,1 \rightclosedint$, if $\zeta_{\cnnb,\radp,\crdp} : \dusnd \to \C$ is defined by
\[
\zeta_{\cnnb,\radp,\crdp}(\eta) 
= \zeta_\cnnb(\eta) \, \chi(|\eta| /\radp) \, \chi(b_1^\eta b_2^\eta / (\crdp |\eta|^2)) ,
\]
then, for all $\alpha \in \leftclosedint 0,\infty \rightopenint^3$ and $\theta,\gamma \in \leftclosedint 0,\infty\rightopenint$ such that $2\gamma > \dim\fst - \min\{1,2\theta\} + 2\sum_{\jtwo=1}^3 (1-2\alpha_\jtwo)_+$,
\[
\| (1+|\cdot|_G)^{-\gamma} \, \Kern_{F(L) \, \zeta_{\cnnb,\radp,\crdp}(\vecU)}\|_1 \leq C_{K,\alpha,\theta,\gamma} \, \radp^{3/2-|\alpha|-\theta} \crdp^{1/4-\alpha_2/2-\alpha_3} \|F\|_{W_2^{|\alpha|}}.
\]
\end{cor}
\begin{proof}
Note that $\zeta_{\cnnb,\radp,\crdp} = \sum_{q \in Y_\crdp} \zeta_{\cnnb,\radp,\crdp,q}$. Since $|Y_\crdp| \lesssim \crdp^{-1/2}$ by Lemma~\ref{lem:hompartuni}(i), the conclusion will follow from Minkowski's inequality if for each summand we can prove the following estimate:
\[
\| (1+|\cdot|_G)^{-\gamma} \, \Kern_{F(L) \, \zeta_{\cnnb,\radp,\crdp,q}(\vecU)}\|_1 
\leq C_{K,\alpha,\theta,\gamma} \, \radp^{3/2-|\alpha|-\theta} \crdp^{3/4-\alpha_2/2-\alpha_3}  \|F\|_{W_2^{|\alpha|}}.
\]
By Proposition~\ref{prp:g743_coneweightedestimate} and H\"older's inequality, this estimate in turn will follow if we show that
\[
\int_G (1+|(x,u)|_G)^{-2\gamma}  \, (1+w(x))^{-2\theta} \prod_{\jtwo=1}^3 (1+|u^q_\jtwo|)^{-2\alpha_\jtwo} \,dx\,du < \infty
\]
(note that the value of the previous integral does not depend on $q$, because the $u^q$ are orthonormal coordinates).
On the other hand, under our hypothesis on $\alpha,\gamma,\theta$, we may decompose $\gamma = \gamma_1+ 2(\gamma_{2,1}+\gamma_{2,2}+\gamma_{2,3})$ so that $2\gamma_1 > \dim \fst -\min\{1,2\theta\}$ and $2\gamma_{2,\jtwo} > (1-2\alpha_\jtwo)_+$ for $\jtwo=1,2,3$. Thus
\begin{multline*}
(1+|(x,u)|_G)^{-2\gamma}  \, (1+w(x))^{-2\theta} \prod_{\jtwo=1}^3 (1+|u^q_\jtwo|)^{-2\alpha_\jtwo} \\
\leq C_{\alpha,\theta,\gamma} \, (1+|x|)^{-2\gamma_1} \, (1+w(x))^{-2\theta} \prod_{\jtwo=1}^3 (1+|u^q_\jtwo|)^{-2\gamma_{2,\jtwo}-2\alpha_\jtwo},
\end{multline*}
and since (by Lemma~\ref{lem:g743_hzestimate}) the right-hand side is integrable over $G$ we are done.
\end{proof}

By choosing $\radp,\crdp$ to be dyadic parameters, we can now sum the estimates corresponding to the first decomposition. In order to do so, the exponents of $\radp$ and $\crdp$ in the estimate must be positive, and this gives further constraints on the choice of $\alpha_1,\alpha_2,\alpha_3,\theta$. Anyhow, a suitable choice of these parameters allows us to obtain for the region $\Omega_\cnnb$ the same $L^1$-estimate obtained in Corollary~\ref{cor:g743_firstweightedestimate_final} for the region $\Omega_\prnb$.

\begin{cor}\label{cor:g743_coneweightedestimate_final}
Let $K \subseteq \R$ be compact. For all smooth $F : \R \to \C$ such that $\supp F \subseteq K$,
for all $\alpha,\beta \in \R$ with $\beta \geq 0$, $2\beta >\alpha+9/2$, $\alpha < -5/2$,
\[
\| (1+|\cdot|_G)^{\alpha} \, \Kern_{F(L) \, \zeta_\cnnb(\vecU)}\|_1 \leq C_{K,\alpha,\beta} \|F\|_{W_2^{\beta}}.
\]
\end{cor}
\begin{proof}
Under our hypothesis, we may choose $\theta$ such that $2\theta \in \leftopenint (9+2\alpha)/4,1 \rightopenint \cap \leftclosedint 0,\beta \rightclosedint$. In particular $-2\alpha > 9-8\theta = 4 - 1 + 2((1-2\theta) + (1-2\theta) + 1)$ and $2\theta < 1$, hence, by Corollary~\ref{cor:g743_coneweightedestimate2},
\begin{equation}\label{eq:g743_partialweightedestimate}
\| (1+|\cdot|_G)^{\alpha} \, \Kern_{F(L) \, \zeta_{\cnnb,\radp,\crdp}(\vecU)}\|_1 \leq C_{K,\alpha,\beta} \, \radp^{1-2\theta} \crdp^{1/4-\theta/2} \|F\|_{W_2^{2\theta}}
\end{equation}
for all $\radp,\crdp \in \leftopenint 0,\infty \rightopenint$. On the other hand, for some $\kappa \in \leftclosedint 0,\infty \rightopenint$,
\[
F(L) \, \zeta_\cnnb(\vecU) = \sum_{\substack{k \in \Z,\, n \in \N \\ 2^{k} \leq 2\kappa\max K}} F(L) \, \zeta_{\cnnb,2^{k},2^{-n}}(\vecU),
\]
hence an estimate for $\|(1+|\cdot|_G)^\alpha \,\Kern_{F(L) \, \zeta_\cnnb(\vecU)}\|_1$ can be obtained via the triangular inequality by summing the corresponding estimates given by \eqref{eq:g743_partialweightedestimate}. The sum converges because both $1-2\theta$ and $1/4-\theta/2$ are positive, and since $2\theta \leq \beta$ the conclusion follows.
\end{proof}

Interpolation with the standard estimate finally allows us to conclude the proof of Proposition~\ref{prp:l1low}.

\begin{prp}\label{prp:g743_l1estimate}
Let $K \subseteq \R$ be compact. For all functions $F : \R \to \C$ such that $\supp F \subseteq K$, and for all $\alpha,\beta \in \R$ such that $\beta \geq 2/2$ and $\beta > \alpha + 7/2$,
\begin{equation}\label{eq:g743_weightedl1}
\| (1+|\cdot|_G)^\alpha \, \Kern_{F(L)} \|_1 \leq C_{K,\alpha,\beta} \| F \|_{W_2^\beta}.
\end{equation}
\end{prp}
\begin{proof}
Analogously as in the proof of Proposition~\ref{prp:2dc_l1estimate}, it is sufficient to prove \eqref{eq:g743_weightedl1} for all $\alpha,\beta$ belonging to either of the following ranges:
\begin{gather}
\label{eq:g743_range1} \beta \geq 0, \qquad \beta > \alpha + 10/2;\\
\label{eq:g743_range2} \beta \geq 0, \quad 2\beta > \alpha + 9/2, \quad \alpha < -5/2;
\end{gather}
the conclusion (i.e., the range $\beta \geq 2/2$, $\beta > \alpha+7/2$) is then obtained by interpolation. On the other hand, the validity of \eqref{eq:g743_weightedl1} in the range \eqref{eq:g743_range1} follows from the standard estimate \eqref{eq:standardl2} and H\"older's inequality. As for the range \eqref{eq:g743_range2}, we decompose $F(\lambda) = F(\lambda) \, \zeta_\prnb(\eta) + F(\lambda) \, \zeta_\cnnb(\eta)$ and then we sum the corresponding estimates for $\Kern_{F(L) \, \zeta_\prnb(\vecU)}$ and $\Kern_{F(L) \, \zeta_\cnnb(\vecU)}$ given by Corollaries~\ref{cor:g743_firstweightedestimate_final} and \ref{cor:g743_coneweightedestimate_final}.
\end{proof}

\bibliographystyle{amsabbrv}
\bibliography{lowdimensional}

\end{document}